\documentclass[letterpaper]{article}
\pdfoutput=1

\usepackage[utf8]{inputenc}
\usepackage{amssymb,amsthm,amsmath,verbatim,graphicx,color,cite,tikz}
\usetikzlibrary{calc}
\usetikzlibrary{decorations.pathreplacing}
\usetikzlibrary{matrix}
\usetikzlibrary{arrows}
\title{A drift approximation for parabolic PDEs with oblique boundary data}
\author{Damon Alexander and Inwon Kim\footnotemark[1]}
\newcommand{\norm}[1]{\|#1\|}

\newcommand{\der}[2]{\frac{\partial #1}{\partial #2}}

\newcommand{\R}{\mathbb{R}}

\newcommand{\e}{\epsilon}
\newcommand{\tikzAngleOfLine}{\tikz@AngleOfLine}
	\def\tikz@AngleOfLine(#1)(#2)#3{%
	\pgfmathanglebetweenpoints{%
	\pgfpointanchor{#1}{center}}{%
	\pgfpointanchor{#2}{center}}
	\pgfmathsetmacro{#3}{\pgfmathresult}%
}
\newtheorem{theorem}{Theorem}[section]
\newtheorem{lemma}{Lemma}[section]
\newtheorem{corollary}{Corollary}[section]
\newtheorem{remark}{Remark}[section]
\DeclareMathOperator{\tr}{tr}

\begin{document}
\maketitle
\begin{abstract}
We consider solutions of a quasi-linear parabolic PDE with zero oblique boundary data in a bounded domain. Our main result states that the solutions can be approximated by solutions of a PDE in the whole space with a penalizing drift term. The convergence is locally uniform and optimal error estimates are obtained. 
\end{abstract}
\section{Introduction}

\footnotetext[1]{UCLA Department of Mathematics.  {\em E-mail addresses:} dalexander@math.ucla.edu, ikim@math.ucla.edu.}

Consider the following parabolic problem with oblique boundary data:
$$
\left\{\begin{array}{lll}
u_t - F(D^2u, Du,u,x,t) = 0 &\hbox{ in }&\Omega\times (0,\infty);\\ \\
D u \cdot \vec{v}(x,t) = 0 &\hbox{ on }&\partial\Omega \times [0,\infty);\\ \\
u(x,0)=u_0(x)&\hbox{ in }& \Omega.
\end{array}\right.\leqno(P_g)
$$
Here $\Omega \subset \R^n$ is a bounded $C^2$ domain, $u_0\in C(\bar{\Omega})$, $\vec{v}\in\R^n$ is smooth, and $F$ is a quasi-linear operator with smooth coefficients given by 
\begin{equation}\label{quasi_linear}
F(D^2u,Du,u,x,t) = \sum_{i,j} q^{ij} (u,x,t) u_{x_ix_j}  + b(Du,u,x,t).
\end{equation}
We use $D$ and $\nabla$ interchangeably to denote the spacial gradient. We assume that $q^{ij}(z,x,t)$ satisfies a uniform ellipticity condition, that is, there exists constants $0<\lambda <\Lambda$ such that for all $(z,x,t) \in \R \times \R \times [0,\infty)$,
\begin{equation}\label{elliptic}
\lambda Id_{n\times n}\le (q^{ij}) \leq \Lambda Id_{n\times n}.
\end{equation}
For a given matrix $M$, we write $M^+$ and $M^-$ to denote its positive and negative parts, that is, $M =M^+-M^-$ with $M^+, M^- \geq 0$ and $M^+M^-=0$.  Using this notation, \eqref{elliptic} is equivalent to the condition 
$$
\mathcal{P}^{-}(M) \leq \sum q^{ij}(z,x,t)M_{ij} \leq \mathcal{P}^{+}(M),
$$
where $\mathcal{P}^{\pm}$ are the extremal Pucci operators defined by 
\begin{equation}\label{pucci}
\mathcal{P}^+(M) := \Lambda \tr( M^+) - \lambda \tr(M^-), \quad 
\mathcal{P}^-(M) := \lambda \tr(M^-) - \Lambda \tr (M^-).
\end{equation}
We also assume that $q^{ij}(z,x,t)$ and $b(p,z,x,t)$ are smooth and 
 \begin{equation}\label{Lipschitz}
 q^{ij}\hbox{ and } b \hbox{ are uniformly Lipschitz with respect to } p,z \hbox{ in } \mathbb{R}^n \times \mathbb{R}.
  \end{equation}

We assume that $\vec{v}(x,t)$ given in the boundary condition of $(P_g)$ is a smooth vector field which satisfies 
\begin{equation}\label{oblique}
\vec{v}(x,t)\cdot\vec \nu(x) \geq c_0,
\end{equation}
for some $c_0 > 0$, where $\vec \nu(x)$ denotes the outward normal vector of $\Omega$ at $x\in\partial\Omega$.

As we show in Appendix A, for a given $\vec{v}(x,t)$ satisfying \eqref{oblique} by possibly adjusting the size of $\lambda$ and $\Lambda$, one can always find a symmetric matrix $A(x,t)$ defined on $\R^n \times [0,\infty)$ that is smooth, satisfies \eqref{elliptic} and  
$$
\vec{v}(x,t) = A(x,t)\cdot\vec \nu(x) \hbox{ on }\partial\Omega.
$$ 
With this representation of $\vec{v}$ using $A$, our goal is to approximate the above problem by introducing a penalizing drift. First let us discontinuously extend $F$ onto all of $\R^n$ by taking 
\begin{align*}
	F(D^2v, Dv, v, x, t) = \left\{
	     \begin{array}{lr}
			F(D^2v, Dv, v, x, t) &\mbox{ if } x \in \Omega \\ \\
			\nabla\cdot ( A(x,t) \nabla v)	& \mbox{ if } x \in \Omega^c.
			 \end{array}
		\right.
\end{align*}
Now consider
$$
\left\{\begin{array}{lll}
v_t - F(D^2v, Dv,v,x,t) - N\nabla \cdot [ v A(x,t) \nabla \Phi] = 0 &\hbox{ in }\R^n\times (0,\infty);\\ \\
v(x,0)=v_0(x) &\hbox{ in } \R^n.
\end{array}\right.\leqno (P_N)
$$
Here $v_0$ is an extension of $u_0$ to $\R^n$ to be defined in \eqref{eqn:v0}. $\Phi$ is a potential whose gradient is zero inside of $\Omega$ and is proportional to the inward normal of $\Omega$ outside of $\Omega$. More precisely, we start with $d(x,\Omega)$ which is $C^2$ provided $x$ is in an outer ball of $\Omega$, and we consider a smooth extension $d(x)$ onto all of $\R^n$ that goes to infinity as $|x| \to \infty$. Then we write
\begin{equation}\label{potential}
\Phi(x) := d(x)^3
\end{equation}
See Theorem~\ref{KK}(a) for the well-posedness of $(P_N)$ with the discontinuous operator $F$. Alternatively one can consider a regularized version of $F$ (see section 4).
\medskip

The approximating problem $(P_N)$ can be viewed in the framework of stochastic particles, where the added drift represents an external force that pushes back the particles which diffused out of the domain $\Omega$. In the context of stochastic differential equations, relevant results have been established in the classical paper of Lions and Snitzman \cite{LandS}, where a similar method of introducing a drift term was used to derive existence of solutions to the Skorokhod problem. 
 
  \medskip
 
  Showing the validity of this approximation is the goal of our paper. Our main result is the following:
 
\begin{theorem}\label{main:thm}
Let $u$ and $v$ respectively solve $(P_g)$ and $(P_N)$ as given above, and let $v_0$ be given by \eqref{eqn:v0}. Then for any $T>0$, $v$ uniformly converges to $u$ in $\bar{\Omega}\times [0,T]$ as $N\to\infty$. Moreover we have
\begin{equation}\label{result}
|v(x,t)-u(x,t)| \leq C N^{-1/3} \hbox{ in } \bar{\Omega}\times [0,T],
\end{equation} 
where $C$ depends only on $n, \lambda, \Lambda, T$ and the regularity of the coefficients and the domain $\Omega$.
\end{theorem}

 While $(P_N)$ is a natural approximation of the original problem $(P_g)$, the convergence result does not appear to be previously proven, even for the case of the heat equation with Neumann boundary data.  Let us briefly discuss the main challenges in the analysis.

\medskip

{\bf Remarks}

1. It is not clear to us whether the above theorem holds with the original $F$ in $(P_N)$ without extending it to have a diffusion term that corresponds to the boundary conditions given $A(x,t)$ outside of $\Omega$. For our analysis this extension was necessary for the rather technical reason of constructing an appropriate barrier of the form  $e^{-N\Phi} f$, based on the stationary solution of the divergence-form equation outside of $\Omega$.

\medskip

2. The rate in \eqref{result} is optimal in some sense for our choice of $\Phi$ in \eqref{potential}, which we show in Section 5.1. $\Phi$ is chosen to have cubic growth for the technical reason that  $\Phi$ then is $C^2$ across $\partial\Omega$. See Theorem~\ref{thm:non-normal} for a result on different choices of potentials.

\medskip

3. The result is limited to  quasi-linear PDEs of the form \eqref{quasi_linear}. This is due to the nature of our argument, which is based on approximating $(P_g)$ by switching the operator $F$ near the boundary of $\Omega$ with  the diffusion operator associated with $A(x,t)$, as explained in the outline of the paper below. To guarantee stability of such an approximation we need uniform regularity of the approximate solutions.  This corresponds to the regularity of parabolic PDEs with leading coefficients discontinuous in one variable; see Theorem~\ref{KK}. It remains open whether the theorem holds for general nonlinear operators that go beyond \eqref{quasi_linear}.

\bigskip
 
$\circ$ {\it Heuristics and difficulties}

\medskip

 For the elliptic case,  arguments from the standard viscosity solution theory were applied in \cite{Barles} to show that the solution of 
\begin{equation}\label{elliptic_approx}
- F(D^2v, Dv,v,x) + N\nabla v \cdot (A(x) \nabla \Phi)= 0 \hbox{ in } \R^n
\end{equation}
uniformly converges to the stationary version of $(P_g)$ for nonlinear, uniformly elliptic $F$. Heuristically, this result can be justified by observing that $N\nabla\Phi$ approximates a singular measure concentrated on the boundary of $\Omega$ with the normal direction, thus leading to the boundary condition $\nabla v \cdot A(x)\vec \nu =0$.  Note that the drift term here has a plus sign in front of them instead of minus in $(P_N)$, which seems to be necessary to carry out the maximum principle-type arguments. For parabolic problems such drift term, pushing density out of the domain along the normal direction, will result in solutions converging to zero in time, which quickly rules out the any direct use of the maximum principle. Moreover, even with the minus sign in front of the drift term, for the parabolic problem the above approximation fails due to additional challenges created by the time variable.  For example, we show in Theorem~\ref{decay} that for $F = \Delta u$, replacing the divergence-form drift term in $(P_N)$ by the non-divergence drift term in \eqref{elliptic_approx} causes the solution to converge to zero over time as $N\to\infty$. On the other hand, the zeroth order term $Nv\nabla\cdot(A(x,t)\nabla\Phi)$ in $(P_N)$ causes a problem in the above heuristics to yield the oblique boundary condition in the limit $N\to\infty$. Thus one concludes that there is a delicate balance between the two terms coming out of the penalizing drift term in $(P_N)$, which must be handled carefully. The main observation that enables our analysis is that the solution of $(P_N)$ outside of $\Omega$ can be bounded by the quickly vanishing barriers of the form $e^{-N\Phi}f$, where $f$ is a smooth function. Our actual argument is built on estimates for the barriers (see section 2.1.1) and does not involve direct estimates on $v$, which suffices for our convergence result, but further asymptotic analysis on $v$ may reveal information on the dynamics of the penalizing drift leading to the boundary condition in $(P_g)$.

%\medskip

%4. Our results extend to the inhomogeneous boundary data $Du\cdot\vec{v}(x)=g(x)$ in $(P_g)$ by extending $g$ 

\bigskip

$\circ$ {\it Outline of the paper}

\medskip

Due to the difficulties described above, we were not able to produce a direct argument to show Theorem~\ref{main:thm}. Instead, we show the theorem first for linear operators where the diffusion matches the boundary flux condition in Section ~\ref{sec:linear}, and then build on these results to address the general case in Section~\ref{sec:general}. 

\medskip

 The general idea in Section \ref{sec:linear} is to use the comparison principle, by testing against barriers created by extending a particular perturbation of the true solution for $(P_g)$. To illustrate the construction of barriers done in Section~\ref{sec:linear}, we first carry out the argument in one dimension in Section~\ref{sec:1d}, in the special case where $F = \Delta u$ and $A = 1$.  Then we proceed to the more general linear case in higher dimensions in Section~\ref{sec:2d}, still in the case where $F$ and $A$ correspond. One important ingredient in the proof is a decomposition argument which eliminates the zeroth order term in the penalizing drift in $(P_N)$, as shown in \eqref{decomposition} and \eqref{decomp_2}.  In Section~\ref{sec:numerics} we will show results from basic numerical experiments which verify the rate of convergence for the heat equation in one dimension.

\medskip

In section~\ref{sec:general}, we introduce an additional approximation to let us utilize the results of the previous section to show the main theorem. Roughly speaking, we will interpolate the diffusion term of $F$ in $(P_g)$ with the one matching the boundary condition near $\partial\Omega$; see $(P_r)$ in Section~\ref{sec:general}. We then consider approximating the modified problem with the penalizing drift term. The important estimate in this section is the uniform rate of convergence between $(P_r)$ and its penalizing approximation $(P_{r,N})$ which is independent of $r$ (see Theorem~\ref{thm_1}), based on the uniform regularity of solutions of $(P_r)$ (see Lemma~\ref{lem:reg}). The uniform regularity estimate for $(P_r)$ draws from the result of Kim and Krylov \cite{KK}, and is of independent interest. We finish with remarks and examples in Section~\ref{sec:remarks}.

\medskip

{\bf Acknowledgements:} 
We thank Jose Carrillo for the interesting discussions which prompted our investigation. We also thank Alexis Vasseur for insightful comments regarding the heuristics above. Both authors are partially supported by NSF DMS 1300445.

\section{PDEs of divergence form}\label{sec:linear}
We first consider the case when $F$ is in linear, in divergence form and matches the co-normal boundary condition, in the following way:  
$$\label{P}
\left\{
     \begin{array}{lll}
	u_t - \nabla \cdot (A(x,t) \nabla u) = 0 &\mbox{in}& \Omega \times (0,\infty); \\
	\nabla u^T A(x,t) \vec \nu = 0 &\mbox{on}& \partial \Omega \times (0,\infty); \\
	u(x,0) = u_0(x) &\mbox{in} &\Omega.
		 \end{array}
	\right.\leqno{(D)}
$$
 For simplicity we rescale so that $\lambda = 1$, so that 
\begin{align} \label{eqn:strictell}
\tag{A1}
	Id_{n\times n} \leq A(x,t) \leq \Lambda Id_{n\times n} \mbox{ for all } x \in \R^n, t \ge 0.
\end{align}
In this case, the approximating problem is written as
$$
\left\{
     \begin{array}{lll}
	v_t - \nabla \cdot [A(x,t) \nabla v] - N\nabla \cdot[  v A(x,t) \nabla \Phi] = 0 &\mbox{in}& \R^n \times (0,\infty); \\
	v(x,0) = v_0(x) &\mbox{on}& \R^n.
		 \end{array}
	\right.\leqno{(D_N)}
$$
Here $\Phi(x)$ is defined in \eqref{potential}, and  $v_0$ is an extension of $u_0$ onto $\R^n$ which will be defined in more detail in Section~\ref{sec:2d}.  We will prove
%Theorem: 2D divergence form------------------------------------
\begin{theorem}\label{thm:neumannapprox2dthm}
Suppose $\Omega$ is $C^2$ and that $A$ is $C^2$, symmetric, and satisfies \eqref{eqn:strictell}.  Then if $u$ solves $(D)$ and $v$ solves $(D_N)$ with initial data $v_0$ given in \eqref{eqn:v0}, we have that
\begin{align*}
	\norm{u - v}_{L^{\infty}(\Omega \times [0,T])} < C(u_0, \Omega, A)T N^{-1/3}.
\end{align*}
\end{theorem}

\subsection{The heat equation in one dimension}\label{sec:1d}
Before handling the problem in multiple dimensions, we illustrate the proof technique on a simpler example, the one dimensional heat equation with Neumann data:
$$\label{eqn:1dheat}
\left\{
     \begin{array}{ll}
	u_t = u_{xx} &\mbox{in } (a, b) \times [0,\infty); \\
	u_x(a,t) = u_x(b,t) = 0 &\mbox{for all } t > 0;  \\
	u(x,0) = u_0(x) &\mbox{for all } x \in [a,b] . 
		 \end{array}
	\right.\leqno{(H)}
$$
We define the approximating problem
$$
\left\{
     \begin{array}{ll}
	v_t = v_{xx} +N v_x \Phi_x + N v \Phi_{xx} &\mbox{in } \R \times (0,\infty); \\
	v(x,0) = v_0(x) &\mbox{for all } x \in \R.
		 \end{array}
		\right.\leqno{(H_N)}
$$
Here $v_0$ is defined as
\begin{equation}\label{initial:H}
	v_0(x) := \left\{
	     \begin{array}{lr}
			u_0(x) & \mbox{ if }  x \in [a,b]\\ 
			u_0(b)e^{-N \Phi(x)} & \mbox{ if } x > b \\
			u_0(a)e^{-N \Phi(x)} & \mbox{ if } x < a,
			 \end{array}
			 \right.
\end{equation}
and $\Phi$ is defined as follows:
\[\Phi(x) := \left\{
     \begin{array}{lr}
	 |x-a|^3 & \mbox{ if } x \le a \\
	0		& \mbox{ if } a < x < b \\
		|x-b|^3 & \mbox{ if } x \ge b.
		\end{array}
	\right.\]
In words, $\Phi$ grows cubically outside the original region, which makes it $C^2$ at the boundary. 

\begin{theorem} \label{thm:heat}
Assume $u_0\in C([a,b])$, and let $u$ and $v$ solve $(H)$ and $(H_N)$ respectively with initial data $u_0$ and $v_0$. Then for any $T>0$, $v$ uniformly converges to $u$ in $[a,b]\times [0,T]$ as $N\to\infty$. In particular we have that for all $T>0$,
\begin{align*}
	\norm{u - v}_{L^\infty([a,b] \times [0,T])}  < C(u_0, a, b) (T+1) N^{-1/3}.
\end{align*}
\end{theorem}
%\textcolor{red}{Can quantify $C = (\norm{u_t(b,\cdot)}_\infty + \norm{u_t(a,\cdot)}_\infty) 10^{4/3} /(b-a)$.  Is that desirable?}\textcolor{blue}{I am not sure if that's right, it also involves bounds for the spacial derivative of $u$ and $Du$ as I understand?} 

For the proof we will perturb the true solution and then extend it to get super- and subsolutions of $(H_N)$ on all of $\R$.  The super- and subsolutions will serve as barriers to show that $v$ is close to $u$ in $\Omega$.  For the specific $v_0$ given by \eqref{initial:H}, the minimal size of the perturbation can be estimated by the barriers and we obtain the rate of convergence. 

\subsubsection*{Building a supersolution}

The first step is to create a supersolution to extend $u$ off $\Omega$, taking the form
\begin{equation}\label{decomposition}
\varphi(x,t) = f(x,t) e^{-N \Phi(x)} .
\end{equation}
Without loss of generality, we will only show the details of the extension to the right of $x = b$.   Note that $\varphi$ satisfies
\begin{align*}
	\varphi_t - \varphi_{xx} - N \varphi_x \Phi_x - N \varphi \Phi_{xx} &= e^{-N \Phi} \Big( f_t - f_{xx} + 2N f_x \Phi_x - N^2 f\Phi_x^2  \\
	&+ N f \Phi_{xx}- N f_x \Phi_x + N^2 f\Phi_x^2  - N f \Phi_{xx}\Big) \\
 &= e^{-N \Phi}\left( f_t - f_{xx} + N f_x \Phi_x\right).
\end{align*}
Thus we need only verify that
\begin{align}\label{eqn:transformeqn}
f_t - f_{xx} + N f_x \Phi_x > 0.
\end{align}
We want $\varphi$ to match $u$ at the boundary, and go above it to the left, that is, $(\varphi_x - u_x)|_{x = b} < 0$.  This would let us create a supersolution extension by taking the infimum of $\varphi$ and $u$. However, since $u_x = 0$ at $b$, this requires $\varphi_x = f_x < 0$ which makes \eqref{eqn:transformeqn} difficult to satisfy.  To avoid this, consider 
\[u_\epsilon(x,t) := u(x,t) + \frac{5\alpha}{b-a}\epsilon (x - (a+b)/2)^2 + \frac{10\alpha}{b-a}\epsilon t.\]
Here $\alpha :=  \norm{u_t(b,\cdot)}_{L^\infty([0,\infty))}$, and $\epsilon$ is a perturbation parameter.  Then $u_\epsilon$ will satisfy the heat equation except with boundary condition $u_{\epsilon,x}(b,t) = 10\alpha \epsilon/2= 5\epsilon \alpha$.

Now we construct $f$ so that it matches $u_\epsilon$ at the boundary.  For simplicity, we assume $b = 0$ and write
\begin{align*}
	f(x,t) &:= u_\epsilon(0,t)  + \alpha \frac{(x - \epsilon)^3 + \epsilon^3}{\epsilon} + \alpha \epsilon x.
\end{align*}
A sample $\varphi$ is shown in Figure~\ref{fig:u_eps}.  The cubic term in $f$ is designed to help for $x$ small, while the linear term will help for larger $x$.  We calculate:
\begin{figure}
\centerline{\includegraphics[width=5in]{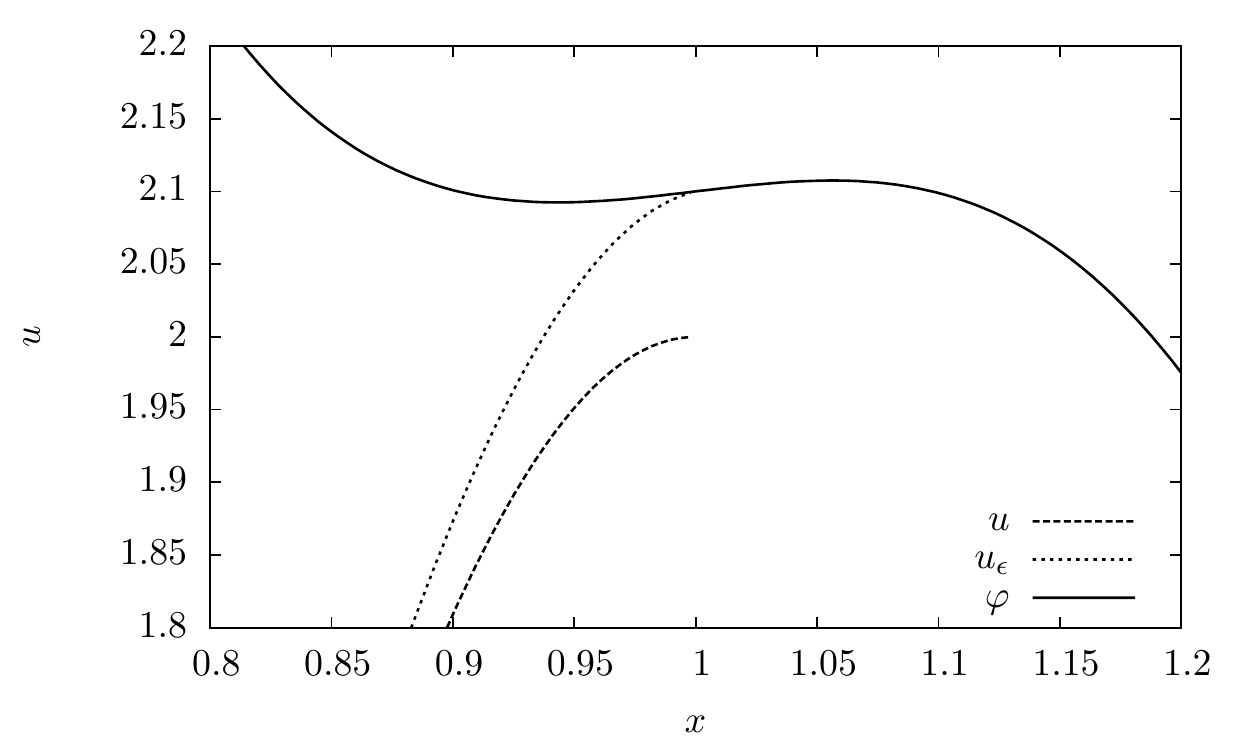}}
	\caption{A sample $u$, $u_\epsilon$, and $\varphi$, where $\Omega = [0,1]$.}
	\label{fig:u_eps}
\end{figure}
 \begin{align*}
 	f_t(x,t) &= u_{\epsilon,t}(0,t)= u_t(0,t) + 10\epsilon \alpha/(0-a), \\
	f_x(0,t) &=  \epsilon (3\alpha + \alpha) < u_{\epsilon, x}(0,t).
 \end{align*}
Then $f_t(x,t) > -2 \alpha$ if $\epsilon < (b-a)/10$.

 Now for $x \in [-\epsilon, \epsilon/2]$, we find
 $$
 	f_x = \alpha \frac{3(x-\epsilon)^2}{\epsilon} + \alpha\epsilon > 0,  \qquad f_{xx} = 6\alpha \frac{x-\epsilon}{\epsilon} \le -3\alpha.
 $$
 Thus we find
 \[f_t - f_{xx} + N f_x \Phi_x > -2\alpha + 3\alpha > 0.\]

 Next, if $x > \epsilon/2$, we find
 $$
  f_x = 3 \alpha(x - \epsilon)^2 / \epsilon + \alpha \epsilon \ge \alpha \epsilon,  \qquad f_{xx} = 6 \alpha(x - \epsilon) / \epsilon < 6 \alpha x / \epsilon.
 $$
 This gives the result
 \begin{align*}
 	f_t - f_{xx} + N f_x \Phi_x &> - 2 \alpha - 6 \alpha x / \epsilon + 3x^2 \alpha \epsilon N \\
	&\ge (N\alpha \epsilon^3 / 4 - 2 \alpha) + \alpha \epsilon^{-1}x \left(2x  \epsilon^2 N - 6  \right).
 \end{align*}
 Then if $N > 8 \epsilon^{-3}$, both terms will be positive in this region, letting us conclude that $\varphi$ is a supersolution of $(H_N)$ on $[-\epsilon, \infty) \times [0,\infty)$.

\subsubsection*{The full supersolution}

Our final supersolution will be as follows:
\begin{equation}\label{supersolution}
	w(x,t) = \left\{
	     \begin{array}{lr}
			u_\epsilon(x,t) &\mbox{ if }  a < x < 0\\
			 \varphi(x,t) & \mbox { if } x \ge 0.
			 \end{array}
		\right.
\end{equation}
This is a supersolution of $(H_N)$ because it can be written as the infimum of a smooth extension of $u_\epsilon$ and $\varphi$.  This works since they touch at $x = 0$ and are ordered appropriately because as shown above,  $u_{\epsilon,x}(0,t) > f_x(0,t)$.   Then for $x > 0$,
\begin{align*}
w(x,0) \ge u_{\epsilon}(0,t) e^{-N \Phi(x)} \ge u_0(0) e^{-N \Phi(x)} = v(x,0).
\end{align*}
Since we can extend $w$ in an analogous way to the left of $a$, applying the comparison principle ensures that $v \le w$ in $\R\times[0,\infty)$ and hence $v \le u_{\epsilon}$ in $[a,b]\times[0,\infty)$.

\subsubsection*{The proof of Theorem~\ref{thm:heat}}

From the supersolution $u_\e$ constructed above setting $N = 10\e^{-3}$, we obtain $v \leq u_\e \leq u + CT\e \leq u+CTN^{-1/3}$. Next, we construct a subsolution as follows. Let 
\begin{align*}
	g(x,t) &= u_{-\epsilon}(0,t)  - \alpha \frac{(x - \epsilon)^3 + \epsilon^3}{\epsilon} - \alpha \epsilon x.
\end{align*}
Then we have that $g_x = - f_x$ and $g_{xx} = -f_{xx}$, and so by similar estimates we find $\psi(x,t) := g(x,t) e^{-N \Phi(x)}$ will be a subsolution on $[-\epsilon, \infty) \times [0,\infty)$.  This lets us extend $u_{-\epsilon}$ to a subsolution $\tilde w$ on all of $\R\times[0,\infty)$.  Then by construction, $\tilde w \le v$ at $t = 0$.  Hence $\tilde w \le v$ for all time by the comparison principle, so in particular $u_{-\epsilon} \le v$ in $[a,b]\times[0,\infty)$.  This lets us conclude that for $(x,t) \in [a,b] \times [0,\infty)$,
\[ u_{-\epsilon}(x,t) \le v(x,t) \le u_{\epsilon}(x,t).\]
Thus provided $N > 10 [(b-a)/10]^{-3}$, we have
\[ \norm{u - v}_{L^\infty(\Omega \times [0,\infty))} < C(u_0, a, b)(T+1) N^{-1/3}. \]

%Lastly, the choice of $v_0 = v_{0,N}$ given in \eqref{eqn:v0} can be generalized in exchange of giving up the rate of convergence. Let us consider $v_0=v_{0,N}$ such that both $v_{0,N}$ and $\{v_{0,N}>0\}$ uniformly converge to $u_0$ and $\{u_0>0\}$, the latter in Hausdorff distance. Then the corresponding solution $v_N$ of $(P_N)$ with the initial data $v_{0,N}$ can be placed between two initial data $w^1_{0,M}< w^2_{0,M}$, both of which and their support converging to $u_0$ as $M\to\infty$. Let $w^1_M$ and $w^2_M$ be the corresponding solutions of the original problem $(P_g)$. Using barriers constructed above, we can show that 
%$$
%w^1_M -CTN^{-1/3}\leq v_N \leq w^2_M + CTN^{1/3} \hbox{ in } \bar{\Omega}\times [0,T].
%$$
%This and the convergence of $w^i_M$ to $u$ as $M\to\infty$ yields the local uniform convergence of $v_N$ to $u$.
\hfill$\Box$

\begin{remark}
 Perhaps the most natural choice for $v_0$ is $v_0 = u_0$ inside $\Omega$ and zero outside.  In this case the convergence rate can be obtained in $L^1$ norm. Observe that for $w$ as given in \eqref{supersolution}, $v\leq w \leq u+ CN^{-1/3}(T+1)$ in $\Omega \times [0,T]$.  Moreover, since $0\leq v\leq w$ one can show that $\int_{\R\backslash\Omega} v(x,t) dx \leq CN^{-1/3}$.
 %\begin{align*}
%	\norm{v}_{L^1(\Omega^c)} &\le C \int_{S^{n-1}} \int_{r(\Theta)}^\infty r^{n-1} e^{-N(r - r(\Theta))^3} \; dr \; d\Theta \\
%	&\le C N^{-n/3} \int_{S^{n-1}} \int_0^\infty (z+C)^{n-1} e^{-z^3} \; dz \; d\Theta = C N^{-n/3}. 
%\end{align*}
The above estimates as well as conservation of mass yields that
\begin{align*}
	\norm{v(\cdot, t) - u(\cdot, t)}_{L^1(\Omega)} \leq CN^{-1/3}(T+1).
\end{align*}

\end{remark}

\subsection{The general linear divergence form equation}\label{sec:2d} % (fold)
Now we consider the divergence form parabolic equation $(D)$, and the approximating problem $(D_N)$.  Generalizing the extension process used in the one-dimensional case requires using the distance function, which is only smooth if we are close to $\Omega$.  To this end, we will require an intermediate domain $\Omega'$ that contains $\Omega$. For $\gamma$ a lower bound on the radius of interior and exterior balls to $\partial \Omega$, we define $\Omega'$ as  
\begin{align}\label{eqn:omegabd}
\Omega' &:= 	\{x : d(x, \Omega) < d_0\}, \\	
	\mbox{ where } 	d_0 &= \frac{1}{2}\min\left[\gamma, \frac{\gamma}{\sqrt{\Lambda^2 - 1}}\left(\Lambda - \sqrt{\Lambda^2 - 1}\right) \right].
\end{align}
Then we have that $d(x,\Omega)$ is $C^2$ inside $\Omega'$, so we can find a $C^2$ extension $d(x)$ that matches $d(x,\Omega)$ inside $\Omega'$, and goes to infinity as $|x| \to \infty$.
Before we prove Theorem~\ref{thm:neumannapprox2dthm}, we prove two lemmas that will help us extend $u$ off $\Omega$.  We define the mapping $S : \Omega' \to \partial \Omega$ to tell us what boundary point our extension takes data from.  We define $S$ in formula as
\begin{align*}
	S(x,t) := x - \tilde d(x,t) A(x,t) \nabla d(x).
\end{align*}
Here $\tilde d(x,t)$ is defined so that $S(x,t) \in \partial \Omega$, and in the case $A = Id_{n\times n}$ simply equals $d(x)$.  In words, $S$ maps $x$ to the closest point in $\Omega$ in direction $-A(x,t) \nabla d(x)$, whereas the closest point is actually in direction $-\nabla d(x)$; see Figure~\ref{fig:S}. 

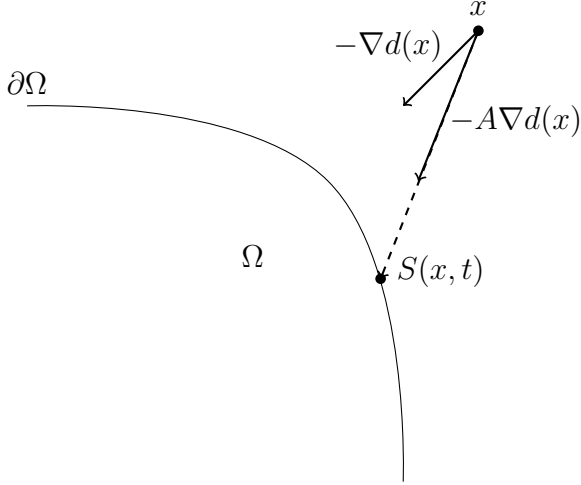
\begin{figure}
\begin{tikzpicture}
\draw plot [smooth, tension = 0.8] coordinates {(5,0) (4,4) (0,5)};
\coordinate (x) at (6,6);
\coordinate (y) at (4.7,2.7);
\fill (x) circle [radius = 2pt];
\draw (3,3) node {\large{$\Omega$}};
\draw  (6,6.3) node {\large{$x$}};
\draw [->, thick, color=black] (x) to (5,5);
\draw (4.8,5.8) node {\large{$-\nabla d(x)$}};
\draw [->, thick, color=black] (x) to (5.2,4.00);
\draw (6.5, 4.8) node {\large{$-A\nabla d(x)$}};
\draw [dashed, ->, thick, color=black] (x) to (y);
\draw (5.5,2.8) node {\large{$S(x,t)$}};
\fill (y) circle [radius = 2pt];
\draw (0,5.3) node {\large{$\partial\Omega$}};

\end{tikzpicture}
\caption{\label{fig:S}An illustration of how $S$ functions}
\end{figure}
Our first lemma shows basic properties of $S$, $d$, and $\tilde d$; the proof employs basic geometry and the implicit function theorem and is deferred to Appendix~\ref{appendixa}.
\begin{lemma}\label{lem:omegalemma}
Suppose that \eqref{eqn:omegabd} holds and $A$ is $C^2$.  Then in $\Omega'$, the distance function $d(x,\Omega)$ is $C^2$, $S(x,t)$ is well defined, and $\tilde d \lesssim d$. Further, for $x \in \Omega' \backslash \Omega$, $A \nabla d |_x \notin T_{S(x,t)} \partial \Omega$.  Lastly, $\tilde d$ is also $C^2$, and hence $S$ is $C^2$ as well, with 
\begin{align}\label{eqn:nablad}
\nabla \tilde d(x,t)^T &= \frac{ \nabla d(S(x,t))^T \left[I - \tilde d(x,t) \nabla A(x,t) \nabla d(x) - \tilde d(x,t) A(x,t) D^2d(x)\right]}{\nabla d(S(x,t)^T) A(x,t) \nabla d(x)}.
\end{align}
\end{lemma}

\begin{lemma}\label{lemma:Slemma}
We can control the size of the directional derivative of $S$ in the drift direction $A \nabla d$ in the following sense:
 \begin{align*}
 	| DS A \nabla d | \bigg|_{(x,t)}\lesssim d(x),
 \end{align*}
 for all $t \ge 0$ and $x \in \Omega' \backslash \Omega$.
 \end{lemma}
\begin{proof}
We remark that as we move slightly in direction $A \nabla d$, $S$ is affected by both $A$ and $\nabla d$ changing.  Since those changes are small, the deflection of $S$ will be proportional to the length we have to travel to get back to $\partial \Omega$, which is proportional to $\tilde d \lesssim d$.

Fix $x$ and $t$.  Define $\eta := A(x,t) \nabla d(x) / |A(x,t) \nabla d(x)|$.  As shown in Lemma~\ref{lem:omegalemma},
\begin{align*}
\nabla \tilde d(x,t)^T \eta &= \frac{ \nabla d(S(x,t))^T \left[I - \tilde d(x,t) \nabla A(x,t) \nabla d(x) - \tilde d(x,t) A(x,t) D^2d(x)\right]}{\nabla d(S(x,t))^T A(x,t) \nabla d(x)} \eta \\
	&= \frac{1}{|A(x,t) \nabla d(x)|} - \tilde d(x,t)\left[ \nabla d(S(x,t))^T \frac{\nabla A(x,t) \nabla d + A(x,t) D^2d(x) }{\nabla d(S(x,t))^T A (x,t) \nabla d(x)}\eta\right] \\
	& =:  \frac{1}{|A(x,t) \nabla d(x)|} + \beta\tilde d(x,t).
\end{align*}
Here $\beta$ is defined this way for brevity and $\nabla A \nabla d$ is the matrix whose $(i,j)$ entry is $(\nabla a^{ij})^T \nabla d$.  We Taylor expand the quantities in $S$ in direction $\eta$ to find
\begin{align*}
  \tilde d(x + h \eta,t) &= \tilde d(x,t) + \frac{h}{|A(x,t) \nabla d(x)|} + \beta\tilde d h+ O(h^2), \\
	\nabla d(x + h \eta) &= \nabla d(x) + h D^2 d \eta + O(h^2), \\
	(A(x+ h \eta,t))  &= A(x,t)+ h (\nabla A  \eta) + O(h^2).
\end{align*}
Calculating the directional derivative directly yields
\begin{align*}
	S(x + h \eta,t) - S(x,t) &= x + h \eta - \tilde d(x+ h \eta,t) A(x + h \eta,t) \nabla d(x+ h \eta) - x + \tilde d(x) A(x) \nabla d(x) \\
	&= h \eta - \left(\tilde d + \frac{h}{|A \nabla d|} + h \beta\tilde d \right) (A + h (\nabla A\eta)) ( \nabla d + h D^2 d \eta)   + \tilde d A \nabla d +O(h^2)\\
	&=  h \left(\eta - \tilde d A D^2 d \eta -\eta - \beta\tilde d \eta -  \tilde d (\nabla A \eta) \nabla d\right)+ O(h^2).
\end{align*}
Dividing by $h$ and taking the limit as $h \to 0$, we see that 
\begin{align} \label{eq:DScalc}
	DS A \nabla d &=  -|A \nabla d| \left(  \tilde d A D^2 d \eta + \beta\tilde d \eta +  \tilde d (\nabla A \eta) \nabla d\right).
\end{align}
Then by Lemma~\ref{lem:omegalemma}, $A \nabla d|_x \notin T_{S(x,t)} \partial \Omega$, that is, it is not tangent to $\partial \Omega$ at the point $S(x,t)$.  Thus by compactness, we find $\nabla d(S(x,t))^T A(x,t) \nabla d(x)$ can be bounded away from zero, and so $\beta < C(A, \Omega) $.  Then factoring out $\tilde d$ from \eqref{eq:DScalc} and using that $\tilde d \lesssim d$ from Lemma~\ref{lem:omegalemma}, we find
\begin{align*}
|	 DS A \nabla d | < C( A, \Omega) \tilde d < C(A, \Omega) d.
\end{align*}
\end{proof}
With this lemma in hand, we are ready to prove Theorem~\ref{thm:neumannapprox2dthm}.  We define $v_0$ as follows:
\begin{align}\label{eqn:v0}
	v_0(x) := \left\{
	     \begin{array}{lr}
		u_0(x)	 & \mbox{in } \Omega\\
		e^{-N\Phi(x)}\mu(x)	u_0(S(x,0))	& \mbox{ in } \Omega^c.
			 \end{array}
		\right.
\end{align}
Here $\mu(x) : \R^n \to [0,1]$ is a smooth function that is one when $d(x,\Omega) < d_0 / 2$ and zero when $d(x,\Omega) > d_0$.  This smoothing factor $\mu$ is necessary since the map $S$ is only defined when $d(x,\Omega) < d_0$.

\subsubsection{Proof of Theorem~\ref{thm:neumannapprox2dthm}}
\begin{proof}[Proof of Theorem~\ref{thm:neumannapprox2dthm}]
We proceed in a similar fashion to the heat equation case, with the difference being more care is required in the extension process.  In particular, the extension used previously now only works on $\Omega'$, and we have to patch it to another solution to create a supersolution on all of $\R^n$.

First, we perturb $u$ to $u_\epsilon$ which has a small positive slope at the boundary.  We proceed by considering the signed distance function 
\begin{align*}
	h(x) &= \left\{
	     \begin{array}{lr}
			 d(x,  \Omega^c)&\mbox{ if } x \in \Omega \\
			-d(x, \Omega) 	& \mbox{ if } x \in \Omega' \backslash \Omega,
			 \end{array}
		\right.
\end{align*}
defined in a neighborhood of $\partial \Omega$ where this is $C^2$.  We extend $h$ to a $C^2$ function $\tilde h$ on all of $\Omega$, and define
\begin{align}\label{approximation}
	u_\epsilon(x,t) &:= u(x,t) + 5\alpha \epsilon\Lambda \left(-\tilde h(x) + \norm{\nabla \cdot (A \nabla \tilde h)}_{L^\infty(\Omega\times [0,\infty))} t + \norm{\tilde h}_{L^\infty(\Omega)} \right),
\end{align}
where $\epsilon > 0$ is a perturbation parameter and $\alpha$ is a constant to be chosen later.  Then $u_\epsilon$ will be a supersolution of $(D)$ satisfying $u_\epsilon \ge u$ and at $\partial \Omega$,
\begin{align*}
	\nabla u_\epsilon^T  A \vec \nu = \nabla u^T  A \vec \nu - 5 \alpha \epsilon\Lambda \nabla \tilde h^T A \vec \nu =5 \alpha \epsilon \Lambda\vec \nu^T A \vec \nu \ge 5\alpha \epsilon\Lambda.
\end{align*}
We look for a supersolution of the modified equation $(D_N)$ on $\Omega' \backslash \Omega$ of the form
\begin{equation}\label{decomp_2}
\varphi(x,t) = f(x,t) e^{-N\Phi(x)}. 
\end{equation}
Let us calculate how this transforms the equation in detail.  We have that
\begin{align*}
	\varphi_t & = f_t e^{-N \Phi}, \\
	\nabla \varphi &= e^{-N \Phi} (\nabla f - N f \nabla \Phi), \\
	-N\nabla \cdot (\varphi A \nabla \Phi) &= e^{-N\Phi} \left(-N\nabla f^T A \nabla \Phi + N^2 f \nabla \Phi^T A \nabla \Phi -N f \sum_{i,j} a^{ij} \Phi_{x_i x_j} - N f \sum_{i,j} a^{ij}_{x_i} \Phi_{x_j} \right), \\
	- \nabla \cdot (A \nabla \varphi) &= - e^{-N \Phi} \sum_{i,j} \Big[ a^{ij}_{x_i} (f_{x_j} - N f \Phi_{x_j}) \Big]\\
	&+ \sum_{i,j} a^{ij}\left[ f_{x_i x_j} - N f_{x_j} \Phi_{x_i} - N f_{x_i} \Phi_{x_j} - N f \Phi_{x_i x_j} +  N^2 f \Phi_{x_j} \Phi_{x_i} \right] \\
	&=  e^{-N \Phi} \left(2 N \nabla f^T A \nabla \Phi - N^2 f \nabla \Phi^T A \nabla \Phi\right) \\
	& +e^{-N \Phi} \sum_{i,j} \left[a^{ij}_{x_i} (N f \Phi_{x_j} - f_{x_j} )+ a^{ij} (N f \Phi_{x_i x_j} -f_{x_i x_j} )\right] .
\end{align*}
Summing these, we find
\begin{align*}
	\varphi_t - \nabla \cdot [A \nabla \varphi - N \varphi A \nabla \Phi] &=  e^{-N \Phi} \bigg( f_t + N \nabla f^T A \nabla \Phi  - \sum_{i,j} [a^{ij}_{x_i}f_{x_j} + a^{ij} f_{x_i x_j}] \bigg) \\
	&=: e^{-N \Phi} (f_t - \mathcal{M} f),
\end{align*}
where $\mathcal{M}$ is defined this way for brevity.  Then we set
\begin{align*}
	f(x,t) := u_{\epsilon}(S(x,t),t) + \alpha\frac{(d(x) - \epsilon)^3 + \epsilon^3}{\epsilon} + \alpha \epsilon d(x),
\end{align*}
where $S(x,t)$ is the mapping onto $\partial \Omega$ defined above.
We calculate:
\begin{align*}
	f_t &= u_{\epsilon,t} + \sum_{i = 1}^n u_{\epsilon,x_i} S_{i,t}, \\
	\nabla f^T &= \nabla u_\epsilon^T DS + 3\alpha\frac{(d(x) - \epsilon)^2}{ \epsilon}  \nabla d^T + \alpha \epsilon \nabla d^T, \\
	f_{x_i x_j} &= \sum_{k,l} \frac{\partial^2 u_\epsilon}{\partial x_l \partial x_k}  \frac{\partial S_l}{\partial x_j} \frac{\partial S_k}{\partial x_i} + \sum_{k}\frac{\partial u_\epsilon}{\partial x_k } \frac{\partial^2 S_k}{\partial x_j \partial x_i} + 6 \alpha\frac{(d(x) - \epsilon)}{\epsilon}d_{x_i} d_{x_j}+ \alpha\left(3 \frac{ (d(x) - \epsilon)^2}{  \epsilon} + \epsilon\right) d_{x_i x_j}, \\
	\nabla \Phi &= 3d(x)^2 \nabla d .
\end{align*}
Then we can find $C(u, A, \Omega)$ so that the following bounds hold for $\epsilon$ small:
\begin{align}
	\bigg|  \sum_{i,j} & a^{ij}\left[\sum_{k,l} \frac{\partial^2 u_\epsilon}{\partial x_l \partial x_k}  \frac{\partial S_l}{\partial x_j} \frac{\partial S_k}{\partial x_i} + \sum_{k}\frac{\partial u_\epsilon}{\partial x_k } \frac{\partial^2 S_k}{\partial x_j \partial x_i} \right] + \sum_{i,j,k} a^{ij}_{x_i} \frac{\partial u_\epsilon}{\partial x_k} \frac{\partial S_k}{\partial x_j} \bigg| \le C(u,\Omega, A), \label{eqn:ubds0}\\
	 u_{\epsilon,t} &= u_t + 5 \alpha \epsilon \Lambda \norm{\nabla \cdot (A \nabla \tilde h)}_{L^\infty (\Omega\times [0,\infty))} \ge -\norm{u_t}_{L^\infty(\partial \Omega\times [0,\infty))} \ge -C(u,A, \Omega),\label{eqn:ubds1}\\
	 \left| \sum_{i = 1}^n u_{\epsilon,x_i} S_{i,t}\right| &= \left|\sum_{i = 1}^n u_{x_i} S_{i,t} - 5 \alpha \epsilon \Lambda \sum_{i = 1}^n \tilde h_{x_i} S_{i,t}\right| \le C(u, A, \Omega) +  \alpha \epsilon C(u, A, \Omega).\label{eqn:ubds2}
\end{align}
Then in particular we find that $f_t \ge -2C - \alpha \epsilon C$, so 
\begin{align*}
f_t - \mathcal{M}f &= f_t  + N \nabla f^T A \nabla \Phi - \sum_{i,j} [a^{ij}_{x_i}f_{x_j} + a^{ij} f_{x_i x_j}] \\
 & \ge -2C - \alpha \epsilon C + N \nabla f^T A \nabla \Phi -C(u, \Omega, A)  - \sum_{i,j} a^{ij}_{x_i} \alpha\left[3  \frac{(d(x)- \epsilon)^2}{  \epsilon} + \epsilon\right] d_{x_j} \\
 &+ - 6\alpha \frac{(d(x) - \epsilon)}{\epsilon} \nabla d^T A \nabla d -\sum_{i,j} a^{ij} \alpha\left[3  \frac{(d(x) - \epsilon)^2 }{ \epsilon }+ \epsilon\right] d_{x_i x_j}.
\end{align*}
Now let us examine the $\nabla f^T A \nabla \Phi$ term more carefully. Applying Lemma~\ref{lemma:Slemma}, we have
\begin{align*}
	\nabla f^T A \nabla \Phi &=  3d(x)^2\left( \alpha[3  (d(x) - \epsilon)^2 / \epsilon + \epsilon] \nabla d^T A \nabla d - \nabla u_\epsilon DS A \nabla d \right) \\
	&\ge 3d(x)^2\left( 3  \alpha (d(x) - \epsilon)^2 / \epsilon +\alpha \epsilon  - C d\right) .
\end{align*}
Then the inner polynomial has a minimum at $d = \epsilon + C \epsilon / 6 \alpha$, where it achieves the value of $\epsilon(\alpha - C^2 / 12 \alpha -C)$.  Thus provided $\alpha \ge 3 C(u, \Omega, A)$, we have $N \nabla f^T A \nabla \Phi \ge 3N \epsilon \alpha d(x)^2 / 2$.  Then this assumption on $\alpha$ yields the bound
\begin{align}\label{eq:Mfbd}
	f_t - \mathcal{M}f &\ge -\alpha - \alpha \epsilon C+ \frac{3}{2}N \epsilon \alpha  d(x)^2 - 6\alpha\frac{ (d(x) - \epsilon)}{\epsilon} \nabla d^T A \nabla d \notag\\
	& - \left[3 \alpha \frac{(d(x)- \epsilon)^2 }{ \epsilon} + \epsilon\right]\sum_{i,j} (a^{ij}_{x_i} d_{x_j} + a^{ij} d_{x_i x_j}). 
\end{align}
Now suppose $ -\epsilon < d(x) < \epsilon/2$. This lets us use the estimate $-(d - \epsilon) \nabla d^T A \nabla d \ge \epsilon/2 $ to bound \eqref{eq:Mfbd} by
\begin{align*}
f_t - \mathcal{M}f &\ge -\alpha + 3 \alpha/2 + O(\epsilon).
\end{align*}
Thus it follows that $\varphi$ is a supersolution here if $\epsilon$ is small.

Next, suppose $ \epsilon / 2 \le d(x) \le d_0$.  Then we simplify \eqref{eq:Mfbd} to get
\begin{align*}
	f_t - \mathcal{M}f &\ge -\alpha - \alpha \epsilon C+ \frac{3}{2}N \epsilon \alpha  d^2 - 6\alpha \Lambda d/\epsilon - C[3 \alpha d^2 / \epsilon + \epsilon] \\
	&= \alpha(Nd^2 \epsilon/2 - 1 - \epsilon C) + \frac{\alpha}{\epsilon}\left(N d^2 \epsilon^2 - 6 \Lambda d - 3C d^2 - C \epsilon^2 \right).
\end{align*}
Then setting $N = 12(Cd_0+\Lambda+1)\epsilon^{-3}$, both terms will be positive, and so $\varphi$ will be a supersolution over all of $\Omega' \backslash \Omega \times [0,\infty)$.

\subsubsection{Creating the full supersolution}
As in the one dimensional case, we define 
\begin{align*}
	w(x,t) &= \left\{
	     \begin{array}{lr}
			 u_\epsilon(x,t) & \mbox{ if } x \in \Omega\\
			 \varphi(x,t)	& \mbox{ if } x \in \Omega' \backslash \Omega.
			 \end{array}
		\right.
\end{align*}
We need to check that $w$ is in fact the infimum of the two supersolutions $u_\epsilon$ and $\varphi$.  This is because at $\partial \Omega$ by construction we have $\varphi(x,t) = u_\epsilon(S(x,t),t) = u_\epsilon(x,t)$, while $\nabla u_\epsilon \cdot A \vec \nu \ge 5 \alpha \epsilon \Lambda$ and
\begin{align*}
	\nabla \varphi \cdot A \vec \nu &= \nabla f \cdot A \vec \nu = \nabla u_\epsilon ^T DS A \vec \nu + 4 \alpha \epsilon \vec \nu^T A \vec \nu \le 4 \alpha \epsilon \Lambda.
\end{align*}
Then since $-A \vec \nu$ points inside $\Omega$, it follows that $\varphi > u_\epsilon$ immediately inside $\Omega$, so $w$ is in fact a supersolution.  Note again that this infimum procedure works even though $u_\epsilon$ is only defined inside $\Omega$, since it crosses $\varphi$ exactly at $\partial \Omega$.  

Now  we extend $w$ again from $\Omega'$ to all of $\R^n$.  To do this, consider a stationary solution
\begin{align*}
  \eta(x) =  2\norm{u_\epsilon}_{L^\infty(\partial \Omega \times[0,\infty))} e^{-N \Phi(x)}.
\end{align*}
Then at $\partial \Omega$, $\eta > w$ and at $\partial \Omega'$, if $\epsilon < d_0 / 2$, we have
\begin{align*}
	\eta(x) &= 2 \norm{u_{\epsilon}}_{L^\infty(\partial \Omega \times [0,\infty))} e^{-N \Phi(x)}, \\
	w(x,t) &= \varphi(x,t) \ge  \alpha \frac{d_0^3}{ 8\epsilon} e^{-N \Phi(x)} .
\end{align*}
Thus $w$ starts off below $\eta$ and provided $\epsilon < d_0^3 \alpha(16 \norm{u_\epsilon}_{L^\infty(\partial \Omega \times [0,\infty))})^{-1}$, $w$ must cross $ \eta$ before $\partial \Omega'$.  Thus by taking another infimum, $w$ can be extended to be a solution on all of $\R^n$. 

Lastly we check the ordering of $w$ versus $v$ at the parabolic boundary.  At $t = 0$, the ordering is clear inside $\Omega$, and since 
\begin{align*}
	f(x,0) \ge u_\epsilon(S(x,0),0) \ge \mu(x)  u_\epsilon(S(x,0),0) = \mu(x) \left[ u _0(S(x,0)) + 5 \alpha \epsilon \Lambda \norm{\tilde h}_{L^\infty(\Omega)} \right ],
\end{align*}
 it follows that $\varphi(x,0) \ge v_0(x)$ outside $\Omega$.  Also, $\eta(x) \ge v_0(x)$ by construction, so $w(x,0) \ge v_0(x)$ as well.  Thus we can apply the comparison principle to deduce that
\begin{align*}
	v(x,t) \le w(x,t) \mbox{ in } \R^n \times [0, \infty).
\end{align*}
Since $w(x,t) = u_{\epsilon}(x,t) \in \Omega$, we see that in $[0,T]$,
\begin{align*}
	v(x,t) - u(x,t) &\lesssim \epsilon (T+1) \lesssim N^{-1/3} (T+1).
\end{align*}

Now we construct the subsolution of $(D_N)$ via a parallel procedure.  As in the one-dimensional heat equation case, define 
\begin{align*}
	g(x,t) &:= u_{-\epsilon}(S(x,t),t) - \alpha\frac{(d(x) - \epsilon)^3 + \epsilon^3}{\epsilon} - \alpha \epsilon d(x).
\end{align*}

This crosses the zero solution inside $\Omega'$, since for $\epsilon$ small, $g(x,t)< 0$ when $d(x, \Omega) = d_0 / 2$.  Then by taking a supremum with the zero stationary solution, we can extend $g(x,t)e^{-N \Phi(x)}$ to a subsolution on all of $\R^n$ that equals $u_{-\epsilon}$ in $\Omega$ and starts below $v_0$.  Thus employing the comparison principle lets us deduce that $v(x,t) \ge u_{-\epsilon}(x,t)$ in $\Omega \times [0,\infty)$, and so we conclude that in $\Omega \times [0,T]$,
\begin{align*}
	|v(x,t) - u(x,t)| \lesssim N^{-1/3} (T+1).
\end{align*}
\end{proof}
% section General Divergence Form Equation (end)

\section{Finite Difference Approximation} % (fold)
\label{sec:numerics}
	In this section we present the results of applying the Crank-Nicolson finite difference method to implement the approximation technique.  First we consider the heat equation on $[0,1]$ with initial data $u_0(x) = \cos(2\pi x) + 1$, which admits solution $u(x,t) = e^{-4\pi^2 t} \cos(2\pi x) + 1$.  We ran the scheme with $300$ spacial panels and $2000$ temporal panels, with zero Dirichlet conditions at the boundary of $\Omega' = [-1,2]$. We ran the experiment for $t \in [0,0.3]$.  The true solution is shown in Figure~\ref{fig:v1DEx2}, and two approximate solutions are shown in Figures~\ref{fig:v1DN100Ex2} to \ref{fig:v1DN1000Ex2}.  
	
Estimates of the rate of convergence as we vary $N$ are shown in Table~\ref{tab:est}, which are close to the analytic result of $N^{-1/3}$.  To remove the inaccuracy of the underlying numerics, for each $N$ we ran the scheme with varying space and time panels, and found that for all $N$ tried, 6400 spacial panels and 102400 time panels ensured the results were stable for fixed $N$.  Note that the $L^\infty$ errors were approximately constant in time, and so it sufficed to use the final error.

\begin{figure}
\centerline{\includegraphics[width=4in]{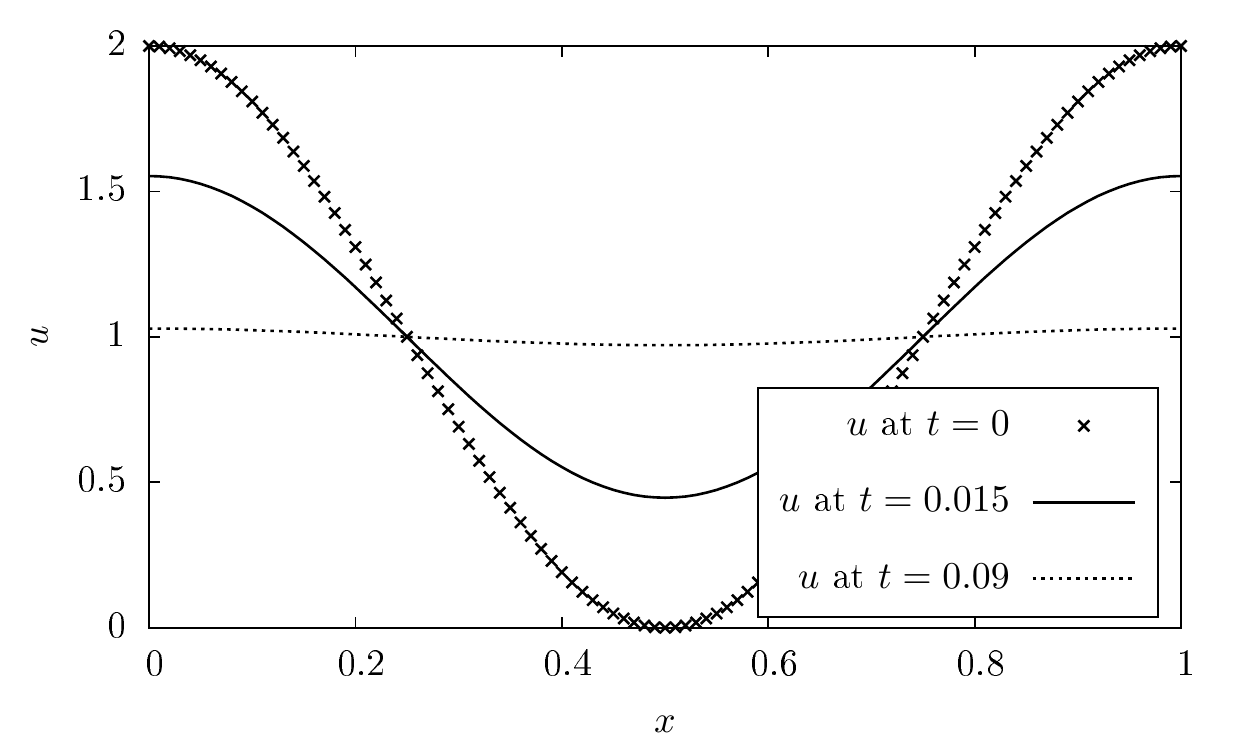}}
\caption{The true solution $u$ at various time points}
\label{fig:v1DEx2}
\end{figure}
\begin{figure}
\centerline{\includegraphics[width=4in]{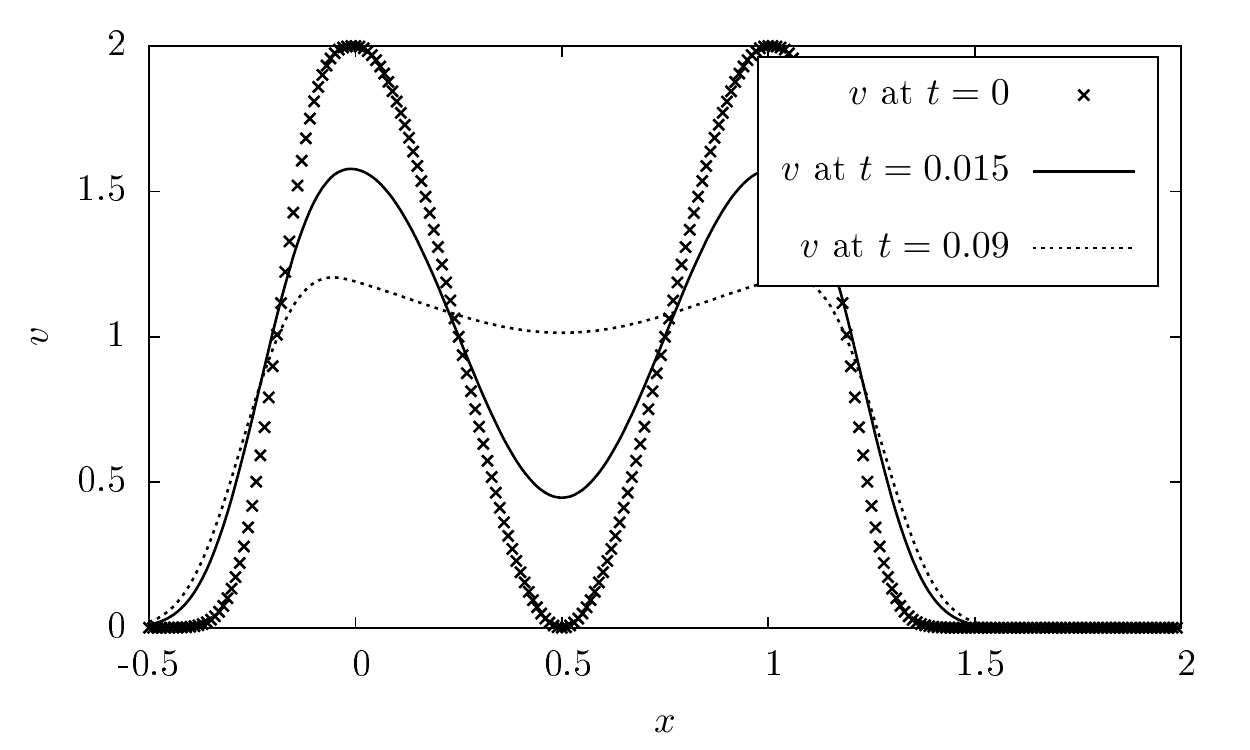}}
\caption{The approximating solution $v$ with $N = 100$}
\label{fig:v1DN100Ex2}
\end{figure}
\begin{figure}
\centerline{\includegraphics[width=4in]{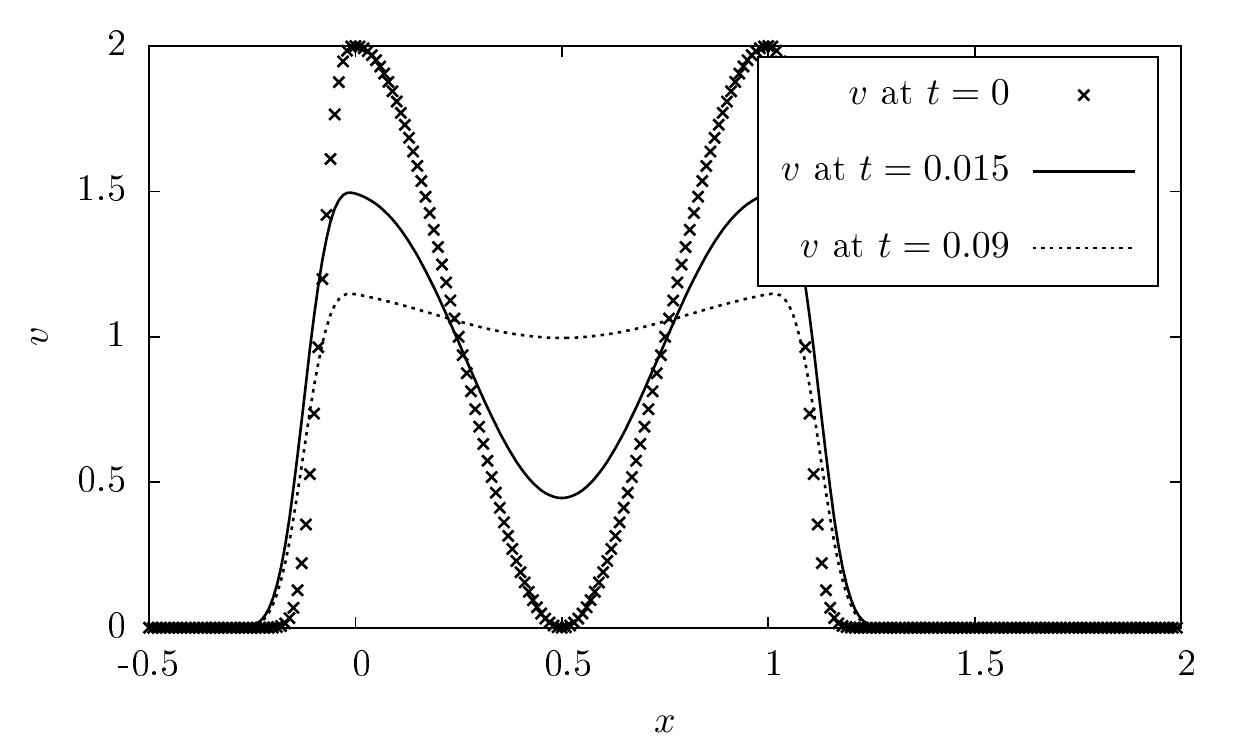}}
\caption{The approximating solution $v$ with $N = 1000$}
\label{fig:v1DN1000Ex2}
\end{figure}

\begin{table}
\centering
\begin{tabular}{l  c  r}
$N$ & $\norm{u-v_N}_{L^\infty([-1,2] \times \{t = 0.3\})}$ & $p$ estimate \\
\hline
$8192$ & $0.0438889$ & $0.2922126913041514$ \\
$16384$ & $0.0356576$ & $0.2996465089723717$ \\
$32768$ & $0.0288452$ & $0.30587833655159546$ \\
$65536$ & $0.0232513$ & $0.3110198856649859$ \\
$131072$ & $0.018688$ & $0.3151992014544745$ \\
$262144$ & $0.0149858$ & $0.31851607677867955$ \\
\hline
\end{tabular}
\caption{\label{tab:est}Estimates of the rate of convergence $N^{-p}$ as we vary $N$}
\end{table}

The big advantage this method has for numerics is that it allows one to run finite difference schemes on domains with general geometry, which normally would require finite element methods.  Since the routine winds up being done over a box, spectral methods can be used.  Also, unlike the finite element methods, this method can be generalized to handle non-linear terms on the interior, at the expense of losing the converge rate estimate.

\section{PDEs of non-divergence form}\label{sec:general}

Building on the previous constructions of barriers, we are now ready to address the general problem given in the introduction in Theorem~\ref{main:thm}. Let $u$ solve $(P_g)$, and let $v$ solve $(P_N)$. To use the barrier argument from the previous section, the boundary operator $F$ must correspond to the operator $\nabla \cdot(A  \nabla v)$ in a neighborhood of $\partial \Omega$. 
%However, in $(P_N)$, we chose a discontinuous extension of $F$ off $\Omega$.  
This necessitates the introduction of two auxiliary problems which feature a regularized operator $F_r$. Towards this, we define
$$
\Omega_r: = \{x\in\Omega: d(x,\partial\Omega)>r\}.
$$ 
Now take a smooth function $f(x)$ which is zero for $x<1$ and one for $x>2$, and write 
$$
g(x):= f(r^{-1}d(x,\partial\Omega)).
$$
Then we define
\begin{equation}\label{interpolation}
F_r(D^2u,Du,u, x,t):= g(x)F(D^2u,Du,u,x,t) + (1-g(x))\nabla\cdot (A(x,t)\nabla u),
\end{equation}
which is smooth, satisfies \eqref{elliptic}, equals $F$ in $\Omega_{2r}$, and equals $\nabla \cdot (A(x,t) \nabla u)$ outside of $\Omega_{r}$. 

\medskip

Next let $w$ and $\tilde v$ solve the two auxiliary problems:
$$
\left\{\begin{array}{lll}
w_t - F_r(D^2w, Dw,w,x,t) =0&\hbox{ in }&\Omega\times [0,T];\\ \\
(A(x,t)\nabla w)\cdot \vec \nu = 0 &\hbox{ on }& \partial\Omega\times [0,T];\\ \\
w(x,0)=u_0(x) &\hbox{ in } & \Omega, \end{array}\right.\leqno (P_r)
$$
and
$$
\left\{\begin{array}{ll}
\tilde{v}_t - F_r(D^2\tilde{v}, D\tilde{v}, \tilde v,x,t) - N\nabla \cdot [\tilde{v} A(x,t) \nabla \Phi] = 0 &\hbox{ in } \R^n \times [0,\infty);\\ \\
\tilde{v}(x,0)=v_0(x) &\hbox{ in } \R^n.
\end{array}\right.\leqno(P_{r,N})
$$

The proof of Theorem~\ref{main:thm} proceeds by showing that as $r \to 0$, the solutions of the auxillary problems $(P_r)$ and $(P_{r,N})$ converge to the solutions of the original problems, uniformly in $N$.  Then since the barriers constructed for the divergence form PDEs apply to the auxiliary problems, this will finish the proof (see Figure~\ref{commutDiag}).

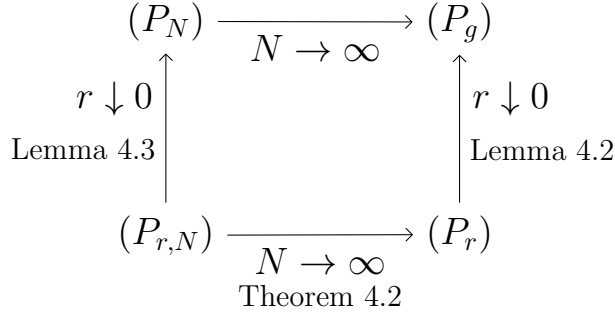
\begin{figure}\label{commutDiag}
\centering
\begin{tikzpicture}\Large
  \matrix (m) [matrix of math nodes,row sep=4em,column sep=5em,minimum width=2em] {
     (P_N) & (P_g) \\
     (P_{r,N}) & (P_r)\\};
  \path[-stealth]
    (m-2-1) edge [->, >=angle 90] node [above left] {$r \downarrow 0$} node [below left, scale = .8] {Lemma~\ref{lem_3}} (m-1-1)
    (m-2-1) edge [->, >=angle 90] node [below] {$N \to \infty$} node [below, yshift = -15, scale = .8] {Theorem~\ref{thm_1}} (m-2-2)
    (m-2-2) edge [->, >=angle 90] node [above right] {$r \downarrow 0$} node [below right, scale = .8] {Lemma~\ref{lem_1}} (m-1-2)
	(m-1-1) edge [->, >=angle 90] node [below] {$N \to \infty$} (m-1-2)
       ;
\end{tikzpicture}
\caption{How the different problems and their solutions relate as we vary the parameters.}
\end{figure}

As a preliminary step, we will develop uniform estimates for $w$ independent of $r$, using the following particular cases of Theorems 2.5 and 2.8 in Kim-Krylov \cite{KK}:

\begin{theorem}\label{KK}
Let us denote $x=(x_1,...,x_n)\in\R^n$ and define
$$
\mathcal{L}u := a^{ij}(x,t)u_{ij} + b(x,t)\cdot Du + c(x,t)u,
$$
where the $\{a^{ij}\}$ satisfy \eqref{elliptic}, are continuous with respect to $(x_1,...,x_{n-1})$, and measurable with respect to the $x_n$ variable. Also, assume $b$ and $c$ satisfy \eqref{Lipschitz}.  Then the following holds for $2< p<\infty$ and $T>0$:
\begin{itemize}
\item[(a)] For any $f \in L^p(\R^n\times [0,T])$, there exists a unique $u\in W^{2,1}_p(\R^n\times (0,T])$ such that $u_t -\mathcal{L}u =f$ in $\R^n \times (0, T]$ with $u(\cdot,0)=0$. Moreover 
$$\|u\|_{W^{2,1}_p(\R^n\times [0,T])} \leq C\|f\|_{L^p(\R^n\times [0,T])}.$$
Here $C$ depends only on $n,p, T, \lambda, \Lambda$, the bounds for $b$ and $c$, and the mode of continuity of the $\{a^{ij}\}$ with respect to $(x_1,...,x_{n-1})$.
\item[(b)] Let $H$ be the half space $\{x=(x_1,...,x_n)\in\R^n: x_n>0\}$ and let $l\in H$. Then for a given $f\in L^p(H\times (0,T])$ there exists a unique $u\in W^{2,1}_p(H\times (0,T])$ satisfying
$$
\left\{\begin{array}{lll}
u_t - \mathcal{L}u = f &\hbox{ in }& H\times (0,T];\\
l\cdot Du = 0 &\hbox{ on }& \partial H \times (0,T];\\
u(x,0)=0 &\hbox { in } H,
\end{array}\right.
$$
with the estimate
$$
\|u\|_{W^{2,1}_p(H\times (0,T])} \leq C\|f\|_{L^p(H\times (0,T])}.
$$
Here $C$ depends only on $n, p, T, l,\lambda, \Lambda$, the bounds for $b$ and $c$,  and the mode of continuity of the $\{a^{ij}\}$ with respect to $(x_1,...x_{n-1})$.
\end{itemize}
\end{theorem}

\medskip

The following lemma is essential to deduce an estimate, uniform in $r$, for the convergence of the solutions of $(P_{r,N})$ to those of  $(P_r)$. 
\begin{lemma}\label{lem:reg}
Let $F(D^2u,Du,u,x,t)$ be given as in \eqref{quasi_linear}, satisfying  \eqref{elliptic} and \eqref{Lipschitz}, and let $w$ solve $(P_r)$. Then for a given $T>0$, the following holds for $0\leq t\leq T$:
\begin{itemize}
\item[(a)] For any $0<\alpha<1$, $w(\cdot,t)$ is uniformly $C^{1,\alpha}$ in $\bar{\Omega}$ with respect to $r$;
\item[(b)] $w_t$ is bounded in $\Omega \times [0,T]$;
\item[(c)] The restriction of $w(\cdot,t)$ on $\partial\Omega$ is uniformly $C^{1,1}$ with respect to $r$.
\end{itemize}
\end{lemma}

\medskip

\begin{proof}

0.  In this proof $C$ denotes various constants which are independent of $r$. Since $w$  is $C^2$ up to $\bar{\Omega}\times (0,T]$, it suffices to get a uniform bound on the derivatives of $w$ with respect to $r$.

\medskip

1. Let us first consider the case when $\vec{v}$ is constant and the domain is a half space, i.e. when 
\[\Omega =H =\{x=(x_1,...,x_n): x_n \geq 0\}.\]
In this case, with compactly supported $u_0$, Theorem~\ref{KK}(b) as well as Morrey's inequality  yield (a) for linear PDEs. For $F$ given as in \eqref{quasi_linear} one can use Schauder's fixed point theorem with the map $\Psi: W^{1,0}_p(\bar{\Omega}_T) \to W^{2,1}_p(\bar{\Omega}_T)$ with sufficiently large $p$, where $u:= \Psi(v)$ solves 
$$
\left\{\begin{array}{lll}
u_t - \sum q^{ij} (v,x,t)u_{x_ix_j} +b(Dv,v,x,t) = 0&\hbox{ in }& \Omega \times (0,T];\\ \\
\vec{v}\cdot Du = 0 &\hbox{ on }& \partial\Omega\times (0,T);\\ \\
u(x,0) = u_0(x) \in C^2(\bar{\Omega}).&&
\end{array}\right.
$$

The argument for smooth, non-constant $\vec{v}$ is parallel to the constant case, which relies on introducing a local change of coordinates to change the problem to a Neumann problem, as written in the proof of Theorem 2.8 in \cite{KK}.  Also, see Remark 2.10 in \cite{KK2}.

\medskip

To generalize from $H$ to $\Omega$, we can apply a local change of coordinates such as in \cite{evans} p. 337-339, which maps $\{x: d(x,\partial\Omega)=r\}$ to $\{x_n=r\}$, to reduce to the half-space case.

\medskip

2.  Note that $W := w_t$ satisfies 
\begin{align*}
\left\{\begin{array}{lll}
W_t -\sum_{i,j} a^{ij}_r(w,x,t)W_{ij} + \partial_p b^r(Dw,w,x,t)\cdot DW + B(x,t)W + C(x,t) = 0 &\hbox{ in }& \Omega \times (0,T];\\ \\
(A(x,t) DW)\cdot \vec \nu = -(A_t(x,t)Dw)\cdot \vec \nu \quad\hbox{ on } \partial\Omega  &\hbox{ on }& \partial\Omega\times (0,T).\\ \\
\end{array}\right.
\end{align*}
We write $a^{ij}_r := g(x)q^{ij} + (1-g(x))a^{ij}$ as the second order matrix of $F_r$ and likewise for $b^r$. Here
 $$
B(x,t) := \sum_{i,j} \partial_z a_r^{ij}(w,x,t)w_{ij}+ \partial_z b^r(Dw,w,x,t),
$$
and 
$$
C(x,t) :=\sum_{i,j}(a_r^{ij})_t (w,x,t) w_{ij}+b^r_t(Dw,w,x,t).
$$
Due to the uniform $W^{2,1}_p$ estimates obtained in step 1, we have that $B$ and $C$ are uniformly bounded with respect to $r$ in $L^p(\Omega\times[0,T])$ for any $2<p<\infty$.  Setting $p=n+1$, \cite{Krylov} yields the uniform $L^\infty$ bound for $W$.
\medskip

3. It remains to show (c). For simplicity, we will only show (c) in the case that $\Omega$ is locally a half space, that is, we will show (c) in $B_{1/2}(0)$ when 
\begin{equation}\label{half_space} 
\Omega\cap B_1(x_0) = \{x\cdot e_1\geq 0\}\cap B_1(0).
\end{equation}
 For  general domains one can take a local change of coordinates as before  to reduce to the half-space case.
 
 \medskip
 
 Let  us choose  a boundary point $x_0\in\partial\Omega \cap B_{1/2}(0)$ and a time $t=t_0$. To show (c), it is enough to show that there exists $M>0$ which is independent of $r$ such that for $x\in\partial\Omega$ and $t\leq t_0$, we have
\begin{equation}\label{parabola}
| w(x,t) - w(x_0,t_0) - Dw(x_0,t_0)\cdot(x-x_0) |\leq M(x-x_0)^2-C(t-t_0)
\end{equation}
in $S:= (\partial\Omega \cap B_{\delta}(x_0)) \times [t_0-\tau,t_0]$. To this end we will build a supersolution and a subsolution of $(P_r)$ and compare it with $w$ in a parabolic neighborhood of $S$. 
\medskip

 We first construct a supersolution of $(P_r)$ to show that in $S$,
\begin{equation}\label{parabola1}
w(x,t) - w(x_0,t_0) - Dw(x_0,t_0)\cdot(x-x_0) \leq M(x-x_0)^2-C(t-t_0).
\end{equation}
Next, fix $\alpha \in (0,1)$. Due to (a) and (b) we know that 
\begin{equation}\label{prelim}
w(x,t) \leq w(x_0,t_0) +  Dw(x_0,t_0)\cdot (x-x_0)+C|x-x_0|^{1+\alpha}-C(t-t_0)
\end{equation} 
in $(B_{\delta}(x_0) \cap \bar{\Omega})\times [t_0-\tau,t_0]$, where $C, \delta$, and $\tau$ are constants  independent from the choice of $r$.
 
 \medskip
 
 Let us denote $(x-x_0)_T:= (x-x_0)-(x-x_0)\cdot e_1$, and consider the function   
 $$
 h(x): = C\delta^{\alpha-1}\left(|(x-x_0)_T|^2  - \frac{\Lambda}{\lambda}(n+1)|(x-x_0)\cdot e_1|^2\right) - C_1\delta^{\alpha} (x-x_0)\cdot e_1 + C C_1\delta^{\alpha} a,
 $$
where $C_1$ is a constant depending only on $A(x,t)$ and $\norm{Dw}_{C^\alpha}$, and $a$ is a constant that is much smaller than $\delta$.  We work on $\Sigma$, a strip neighborhood of $x_0$ which is narrow in the direction of $-e_1 = \vec \nu$:
$$
\Sigma:=B_{\delta}(x_0) \cap\{|(x-x_0)\cdot \vec \nu| \leq a\}\cap \Omega.
$$
Now let us define
$$
 \tilde{h}(x,t):= h(x) +w(x_0,t_0)+ Dw(x_0,t_0)\cdot(x-x_0) -C_2(t-t_0)
 $$ 
where $C_2$ is to be chosen later.  We claim that $\tilde{h}$ satisfies 
\begin{equation}\label{parabolic}
\tilde{h}_t - F_r(D^2\tilde{h}, D\tilde{h},\tilde h, x, t) \geq 0 \hbox{ in } \Sigma\times [t_0-\tau, t_0] 
\end{equation}
if $\delta$ and $\tau$ are chosen small enough, but independently of $r$.  To see this, observe that
$$
\mathcal{P}^+(D^2h) \leq - 2 \Lambda C\delta^{\alpha -1}.
$$
This fact along with \eqref{elliptic} and \eqref{Lipschitz} yields 
 $$
\begin{array}{lll}
\tilde{h}_t - F_r (D^2\tilde{h}, D\tilde{h}, \tilde{h},x,t)  &\geq& -C_2 + g(x)b(D\tilde{h}, \tilde h,x,t) +2 \Lambda C\delta^{\alpha-1} \\ 
 &\geq & 2\Lambda C\delta^{\alpha-1} -O(|D\tilde{h}| +|\tilde{h}| + |x-x_0| + |t-t_0|).
 \end{array}
$$ 
From the definition of $\tilde{h}$ one can check that, in $\Sigma\times [t_0-\tau,t_0]$ with small $\tau$, $|\tilde{h}| \leq C$ and
\begin{equation}\label{gradient}
 |Dh(x)| \leq 2 C \delta^\alpha + 2\frac{\Lambda}{\lambda} C\delta^{\alpha-1} a - C_1 \delta^\alpha \leq C\delta^{\alpha}.
\end{equation}
Hence due to (a),  $|D\tilde h| \leq |Dw| + |Dh|\le C+C\delta^{\alpha}$, and thus $|D\tilde{h}|, |\tilde{h}|\leq C$.  Thus we conclude that $\tilde{h}$ satisfies \eqref{parabolic} if $\delta$ and $\tau$ are sufficiently small.  

\medskip

Moreover, due to \eqref{oblique} we have 
$$
\begin{array}{lll}
[A(x,t)D\tilde{h}(x,t)]\cdot (-e_1) &=& [A(x,t) (Dh + Dw(x_0,t_0))] \cdot (-e_1) \\ \\
&= & [A(x,t) (Dh + Dw(x,t) + [Dw(x_0,t_0) - Dw(x,t)])] \cdot (-e_1) \\ \\
& \ge & [A(x,t) (C \delta^{\alpha - 1} (x-x_0) - C_1 \delta^\alpha e_1) ] \cdot (-e_1) - \norm{A} \norm{Dw}_{C^\alpha}\\ \\
&\ge & C_1 \delta^{\alpha} \lambda -  \norm{A} \norm{Dw}_{C^\alpha} - \norm{A}{C \delta^\alpha} \\ \\
&\geq  & 0 \hbox{ on } (\partial\Omega\cap \Sigma)\times [t_0-\tau, t_0],
\end{array}
$$
where the last inequality holds if $C_1$ is chosen sufficiently large with respect to $\|A\|$ and $\norm{Dw}_{C^\alpha}$.

\medskip

The above arguments let us conclude that $\tilde{h}$ is a supersolution of $(P_r)$ in $\Sigma\times [t_0-\tau,t_0]$.  Now let us compare $w$ and $h$ on the parabolic boundary of the domain $\Sigma\times [t_0-\tau,t_0]$ excluding $\partial\Omega$.  Observe that on $(\partial\Sigma \backslash \partial\Omega) \times [t_0-\tau,t_0]$, due to \eqref{prelim} we have
\begin{align*}
\tilde{h}(x,t) -w(x,t)&\geq C\delta^{\alpha-1}|(x-x_0)_T|^2 - C|x-x_0|^{1+\alpha}  - C_1\delta^{\alpha}|(x-x_0)\cdot e_1| +  \\
&+  C\delta^{\alpha}a- C\delta^{\alpha - 1}\Lambda (n+1) |(x-x_0)\cdot e_1|^2 / \lambda   \\ 
  &\geq 0,
  \end{align*}
since $|(x-x_0)_T| \sim \delta$ and $|(x-x_0)\cdot e_1| \leq a$ on $\partial\Sigma \backslash \partial\Omega$. Moreover
$$
\tilde{h} - w(x,t) \geq C_2\tau -C\delta^{\alpha} \geq 0 \hbox{ on } \Sigma \cap\{t=t_0-\tau\},
$$
if $\tau$ is chosen to be larger than $\delta^{\alpha}$ and if $C_2$ is sufficiently large. 
Therefore, we conclude that $w\leq \tilde{h}$ in $\Sigma \times [t_0-\tau,t_0]$, which yields \eqref{parabola1}. 

\medskip

A parallel argument can be used to provide the lower bound.

\end{proof}
Let us point out that  the barrier constructed in the proof of  Theorem~\ref{thm:neumannapprox2dthm} only relied on the $C^{1,1}$ spacial bounds of $u$  and $L^\infty$ bounds of $u_t$ restricted to  $\partial\Omega$ to get the bounds given in Equations~\eqref{eqn:ubds0}-\eqref{eqn:ubds2}. This is because the supersolution constructed in Theorem 2.1 was built off the behavior of the true solution $u$ along the boundary. In addition, the space-time $C^1$ bound on $u$ lets us use Taylor series to show that this supersolution in fact has the right ordering against $u_\epsilon$ at the boundary. Thus Lemma~\ref{lem:reg} lets us create a supersolution extension on the full domain that is uniformly close to $w$ with respect to $r$, by taking an infimum of the candidate function and the (perturbed) true solution, which cross at $\partial \Omega$.   We can then apply the comparison principle to $\tilde{v}$ and $w$ as before and conclude:

\begin{theorem}\label{thm_1}
Let $\tilde{v}$ and $w$  respectively solve $(P_{r,N})$ and $(P_r)$. Then for any $T>0$ we have 
\begin{equation}\label{convergence_111}
\sup_{\Omega \times [0,T]} |\tilde{v} -w| \leq C(\Omega, \Lambda,\lambda, T) N^{-1/3}
\end{equation}
\end{theorem}
 Next we show by barrier arguments the following:
\begin{lemma}\label{lem_1}
Let $u$ solve $(P_g)$ and $w$ solve $(P_r)$. Then 
$$
 \|u - w\|_{L^\infty(\Omega \times [0,\infty))} \leq C(\Omega, u_0,  c_0, \frac{\Lambda}{\lambda})(1+ e^{Lt})r ,
 $$
where $L$ is the Lipschitz constant for $F$ given in \eqref{Lipschitz}.
% \end{itemize}
\end{lemma}

\begin{proof}
1. Before we begin, we remark that we will later require that the spacial operators $F$ and $F_r$ be decreasing in the zero-th order term.  We can assume this by applying the transform $U := e^{-Lt} u, W := e^{-Lt} w$.  Then we find that $U$ satisfies
\begin{align*}
	U_t &= e^{-Lt} u_t - L U = e^{-Lt} F(D^2 u, Du, u,x,t) - LU \\
	&= e^{-Lt} F(e^{Lt} D^2 U, e^{Lt} DU, e^{Lt} U,x,t) - LU \\
	&=: G(D^2 U, DU, U,x,t),
\end{align*}
where now $G$ is still uniformly elliptic for $t \in [0,T]$ and in particular is decreasing in the $U$ argument.  Likewise, $W$ satisfies
\begin{align*}
	W_t &= e^{-Lt} F_r(e^{Lt} D^2 W, e^{Lt} DW, e^{Lt} W,x,t) - LW \\
	&=:G_r(D^2 W, DW, W,x,t).
\end{align*}
Denote the problems that $U$ and $W$ solve $(\tilde P_g)$ and $(\tilde P_r)$.  Note that $U(\cdot,t)\in C^{2,\alpha}(\bar{\Omega})$ due to  \cite{LT}. 

\medskip

2. Let us define  
$$
M:= \max_{0<r<1,x\in\Omega\backslash \Omega_{2r}, 0\leq t\leq T} (\max(G_r(D^2U, DU,U,x,t),1)),
$$
which is independent of $r$ and finite due to the regularity of $U$.
Let $d_r(x)$ denote the distance function $d(x,\Omega_{2r})$ and its smooth extension by $\tilde{d}(x)$, where $|\tilde{d}| \leq 1$ and $\tilde{d}(x)=d_r(x)$ in a small neighborhood of $\Omega \backslash \Omega_{2r}$.  Let 
$$
C_0 := 2\frac{\sup_{\partial\Omega \times [0,T]}|\vec v(x,t)|}{\lambda c_0},
$$ where $c_0$ is given in \eqref{oblique}. Now consider 
$$
w_2(x,t):=U(x,t)+ 2C_0Mr + C_0Mr\tilde{d}(x)\hbox{ in } \Omega.
$$
 
   \medskip

Note that on any level set of $d_r$ in $\Omega\backslash \Omega_{2r}$ the sum of the tangential second derivatives of $d_r$ amounts to the mean curvature of its level set and the normal second derivative of $d_r$ is zero. Thus, due to \eqref{Lipschitz}, given that $r$ is small enough, 
\begin{equation}\label{second}
G_r(D^2 w_2, Dw_2,w_2,x,t) \leq G_r(D^2U, DU, U,x,t) + O(Mr[\|D^2\tilde{d}\|_{L^\infty}+\|D\tilde{d}\|_{L^\infty}+1]) \hbox{ in } \Omega \backslash \Omega_{2r}.
\end{equation}

From \eqref{second} and the fact that $Dd_r = \vec \nu(x) +O(r)$ on $\partial\Omega$ we deduce that 
$$
\left\{\begin{array}{lll}
(w_2)_t - G(D^2w_2,Dw_2, w_2,x,t) \geq -C_1Mr, &\hbox{ in } &\Omega_{2r}\times (0,T]; \\ \\
(w_2)_t - G_r(D^2w_2, Dw_2,w_2,x,t) \geq -M-C_1Mr &\hbox{ in }&(\Omega \backslash \Omega_{2r})\times (0,T];\\ \\
\nabla w_2 \cdot \vec{v}(x,t) \geq 2 \sup |\vec v| Mr/\lambda &\hbox{ on }&\partial\Omega\times (0,T];\\ \\
w_2(x,0)=U(x,0) + 2C_0Mr+C_0Mr\tilde{d}(x) \geq U(x,0)+C_0Mr &\hbox{ in }& \Omega.
\end{array}\right.
$$
Here $C_1$ is a constant independent of $r$. Now define 
$$
h(x):= Md_r^2(x).
$$ 
We will show that the function defined by 
\begin{equation}\label{barrier2}
\tilde{w}(\cdot,t):= \left\{\begin{array}{lll}
w_2 + C_1M rt &\hbox{ in }& \Omega_{2r}\\
w_2 -\lambda^{-1}h+C_1Mrt& \hbox{ in } &\Omega \backslash \Omega_{2r}
\end{array}\right.
\end{equation}
 is a supersolution of $(\tilde P_r)$ if $r$ is sufficiently small. To this end we develop estimates on $h$. Note that $D^2 d_r$ is bounded in $\Omega \backslash \Omega_{2r}$ since $\partial \Omega_r$ is $C^2$ for $r$ small. From these facts and that $d_r \leq 2r$ in $\Omega \backslash\Omega_r$, it follows that 
$$
D^2(d_r^2) = 2d_r D^2 d_r + 2Dd_r (Dd_r)^T \geq 2Dd_r(Dd_r)^T - O(r)(Id_{n\times n}).
$$
Thus since $|Dd_r| = 1$, we have
$$
(D^2(d_r^2))^ + \ge 2-O(r) \hbox{ and } (D^2(d_r^2))^- \leq O(r).
$$
It follows from the uniform ellipticity of $G_r$ with respect to $r$ that $h$ satisfies
$$
\left\{
\begin{array}{lll}
\mathcal{P}^{+}(D^2h) \geq 2\lambda M- \Lambda O(r) &\hbox{ in }& \Omega \backslash \Omega_{2r};\\ \\
|Dh|  = M |Dd_r| d_r \leq 2Mr   &\hbox{ on }& \partial\Omega ;\\ \\
0\leq h\leq 16Mr^2&\hbox{ in }& \Omega \backslash \Omega_{2r}.
\end{array}\right.
$$
Then since $G_r$ is decreasing in the zero-th term, we find that in $\Omega_{2r}$,
\begin{align*}
	\tilde w_t &= C_1 M r + (w_2)_t \ge G_r(D^2 w_2, D w_2, w_2,x, t)\\
	 &\ge G_r(D^2 \tilde w, D \tilde w, \tilde w, x , t).
\end{align*}
On the other hand, in $\Omega \backslash \Omega_{2r}$ we find
\begin{align*}
	\tilde w_t &= C_1 M r + (w_2)_t \ge G_r(D^2 w_2, D w_2, w_2,x, t) - M \\
 &\ge G_r(D^2 (w_2 - \lambda^{-1} h), D w_2, w_2, x, t) + 2M - M - O(r) \\
 &\ge G_r(D^2 \tilde w, D \tilde w, \tilde w, x, t) + M - O(r) .
\end{align*}
Next, 
 \[
 (Dw_2 -\lambda^{-1} Dh) \cdot \vec v(x,t) \ge 2\lambda^{-1} \sup|\vec v| Mr-2Mr |\vec v|\lambda^{-1}  \ge 0.
 \]
Thus $\tilde w$ is a supersolution of $(\tilde P_r)$ in $\Omega \times (0,T]$ if $r$ is small.
Moreover, if $r$ is sufficiently small,
$$
\tilde{w}(x,0) \geq U(x,0)+ Mr - 16\lambda^{-1}Mr^2 \geq U (x,0) = W(x, 0).
$$
Thus it follows from the comparison principle for solutions of $(\tilde P_r)$ that 
$$
W\leq \tilde{w}\hbox{ in } \Omega \times [0,T],
$$
Then computing
\begin{align*}
	e^{-Lt} w \le \tilde w = e^{-Lt} u + 2C_0 Mr + C_0 Mr\tilde d(x) - \lambda^{-1}h + C_1 M r t
\end{align*}
shows that $w\leq u+Ce^{LT}r$ in $\Omega \times [0,T]$.

A lower bound can be obtained with parallel arguments.
\end{proof}

\begin{corollary}\label{cor_1} For $\tilde v$ solving $(P_{r,N})$ and $u$ solving $(P_g)$,
$$
|\tilde{v} -u| \leq C(T)(r +N^{-1/3}) \quad\hbox{ in } \Omega \times [0,T].
$$
\end{corollary}

Lastly we show the following:
\begin{lemma}\label{lem_3}
 For fixed $N$, $\tilde{v}$ locally uniformly converges to $v$ in $\bar{\Omega}\times [0,\infty)$ as $r\to 0$.
 \end{lemma}
 
\begin{proof}
 
 Let $v_r = \tilde{v}$ be the solution of $(P_{r,N})$ associated with $F_r$. Since $N$ is fixed, $v_r$ is uniformly $C^{1,\alpha}$ in space with $\alpha > 1/2$ and hence has a subsequential limit we denote by $v_0$. We claim $v_0$ is a viscosity supersolution of $(P_N)$; the subsolution case is analogous.  Suppose it is not, and so we can find a smooth function  $\varphi$ that crosses $v_0$ from below at some point $(x_0,t_0)$ that satisfies
 \begin{align}\label{phiest}
 \varphi_t - F(D^2 \varphi, D\varphi,x,t ) - N \nabla \cdot [ \varphi A(x,t) \nabla \Phi] < -\delta < 0,
 \end{align}
 and by smoothness of $\varphi$ and $F$ we can assume this holds in a neighborhood of $(x_0, t_0)$.  
  
  Note that we must have $x_0 \in \partial \Omega$ to not get an immediate contradiction, since otherwise $v_0$ and $v_r$ have the same equation for $r$ small. Then we can find points $(x_r, t_r) \to (x_0,t_0)$ where $\varphi - v_r$ has a local maximum with value $z_r$.  These points must all lie in $\Omega\backslash\Omega_{2r}$, as outside $v_r$ and $v_0$ satisfy the same equation.  The goal is to push the crossing point into $\Omega_{2r}$.  Using $\gamma$ as the minimial radius of interior balls of $\Omega$, let 
   \begin{align*}
\varphi_r := \varphi - \frac{\delta |x - x_r|^2}{4 \lambda n} + \delta(t-t_r)/4.
 \end{align*}  
 Then $\varphi_r$ will still be a subsolution near $(x_0, t_0)$, but now $\varphi - v_r$ has a strict local maximum at $(x_r, t_r)$.

Next, consider the region $B_{\sqrt{\gamma r}}(x_r) \times [t_r - \tau, t_r]$.  Then if $r$ and $\tau$ are small enough, this region will be contained in the region where $\varphi - v_0$ has a local maximum and \eqref{phiest} holds.  We are going to use the fact that $v_r$ is uniformly $C^{1,\alpha}$ in space, independent of $r$.  For $r$ small, $x_r$ must lie within distance $\gamma$ of $x_0$, in which case it has a unique nearest boundary point we denote by $y_r$.  For $\vec \nu_r$ the outward unit vector at $y_r$, let 
\[ 
 h(x,t) := \varphi_r(x,t) - 20Cr^{\alpha}(x-x_r)\cdot \vec \nu_r,
 \]
  where $C$ is larger than the sum of the uniform $C^{1,\alpha}$ norms of $v_r$ and $\varphi_r$.  Next, since $D\varphi_r = Dv_r$ at $(x_r,t_r)$, by the uniform $C^{1,\alpha}$ regularity of $v_r$ we have 
\begin{align*}
	(\varphi_r - v_r)(x,t_r) &= \int (D \varphi_r - D v_r)(s,t_r)\cdot ds + z_r\\
	&= \int [(D \varphi_r - D v_r)(s,t_r)-(D \varphi_r - D v_r)(x_r,t_r)]\cdot ds + z_r\\
	&\ge -\int  (\norm{v_r}_{C^{1,\alpha}}+\norm{\varphi_r}_{C^{1,\alpha}}) |s-x_r|^\alpha ds  + z_r\\
	&\ge -C |x - x_r|^{1+\alpha} + z_r.
\end{align*} 
This lets us compute that
\begin{align*}
	(h - v_r)(x_r - 5r \vec \nu_r, t_r) &\ge -5^{1+\alpha}C  r^{1+\alpha} + z_r + 20\cdot 5Cr^{1+\alpha}\\
	&\ge  z_r + 75 C r^{1+\alpha} .
\end{align*}
On the other hand, if $(x-x_r) \cdot \vec \nu_r \ge -3r$, we have
\begin{align*}
	(h - v_r)(x, t_r) &\le (\varphi_r - v_r)(x,t_r) + 20 \cdot 3\cdot Cr^{1+\alpha}  \\
	&\le z_r + 60 C r^{1+\alpha}.
\end{align*}
Thus the maximum of $h - v_r$ in $B_{\sqrt{\gamma r}}(x_r) \times [t_r - \tau, t_r]$ occurs in $\{ (x - x_r) \cdot \vec \nu_r \le - 3r\}$.  Further, if $r$ is small enough, it must occur on the parabolic interior because on the spacial edge of the parabolic boundary,
\[h - v_r \le -\frac{\gamma \delta r}{4 \lambda n } -20 C r^\alpha (x - x_r) \cdot \vec \nu_r \le -\frac{\gamma \delta r}{4 \lambda n } +40 C \gamma r^{\alpha + 1/2} \le 0.
\]
This is because $\alpha> 1/2$, and on the temporal edge, $\varphi_r$ was not modified.  
	 
Now it remains to show that this maximum occurs inside $\Omega_{2r}$.  This is because of the square root scaling we used.  That is, $y_r$ must have an interior ball $B_\gamma(x_r') \subset \Omega$ that contains $x_r$.  Then $\Omega_{2r}$ must contain $B_{\gamma - 2r}(x_r')$.  We assume for simplicity the worst case scenario where $x_r = y_r$, that is, $x_r$ is on $\partial \Omega$.  Now we show that
\begin{align*}
B_{\sqrt{\gamma r}}(x_r) \cap \{(x - x_r) \cdot \vec \nu_r \le - 3r\} \subset B_{\gamma - 2r}(x_r').
\end{align*}
This follows because the hyperplane $\{(x-x_r) \cdot \vec \nu_r \le -3r\}$ intersects $B_{\sqrt{\gamma r}}(x_r)$ with width $\sqrt{\gamma r - 9 r^2}$ and it intersects $B_{\gamma - 2r}(x_r')$ with width
\[ \sqrt{(\gamma - 2r)^2 - (\gamma - 3r)^2} = \sqrt{2r \gamma - 5 r^2},\]
as can be seen from Figure~\ref{fig:xrpic}.  This width is larger provided $r < \gamma /4$, and by the definition of $\vec \nu_r$, the hyperplane is perpendicular to the line between $x_r$ and $x_r'$.  

This entails that $\varphi_r - v_r$ has a local maximum inside $\Omega_{2r}$, where $\varphi_r$ is a subsolution of the same equation as $v_r$.  Thus we are done, since this contradicts that $v_r$ was a viscosity solution.

%picture of circles used to prove v_r to v lemma
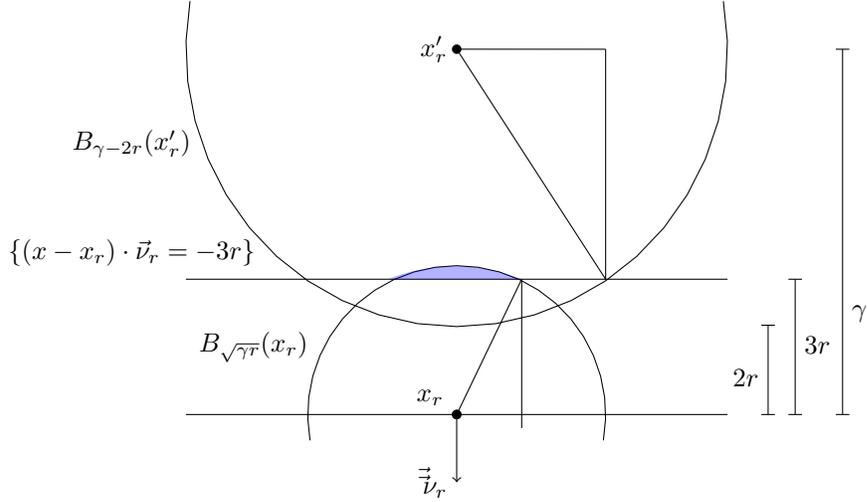
\begin{figure}
\centering
\begin{tikzpicture}[scale = 1.8]
\coordinate (x1) at (-2, 0);
\coordinate (x2) at (2, 0);
\coordinate (y1) at (-2, 1);
\coordinate (y2) at (2, 1);
\coordinate (xr) at (0,0);
\coordinate (xp) at (0,2.7);
%\draw[style=help lines] (-2,0) grid (2,3);
%draw nodes
\fill (0,0) circle [radius = 1pt];
\fill (xr) circle [radius = 1pt] node [above, xshift = -10] {$x_r$};

%draw nu
\draw [->] (xr) -- (0, -.5) node [left] {$\vec \vec \nu_r$};
%circle around x_r
\draw plot [domain=-10:190] ({1.1*cos(\x)}, {1.1*sin(\x)});
%\draw[-] (2pt, 0pt) -- (-2pt, 0pt) node [below] {$x_r$};
\draw[-] (x1) -- (x2);
\draw[-] (y1) -- (y2);
%draw length side bar things
\draw[-] (2.5, 0) -- (2.5, 1) node [right, yshift = -25] {$3r$};
\draw[-] (2.45,0) -- (2.55,0);
\draw[-] (2.45,1) -- (2.55,1);

\draw[-] (2.3, 0) -- (2.3, .66) node [left, yshift = -20] {$2r$};
\draw[-] (2.35,0) -- (2.25,0);
\draw[-] (2.35,.66) -- (2.25,.66);

\draw[-] (2.85, 0) -- (2.85, 2.7) node [right, yshift = -100] {$\gamma$};
\draw[-] (2.8,0) -- (2.9,0);
\draw[-] (2.8,2.7) -- (2.9,2.7);
%circle around xp
\draw plot [domain=170:370] ({2*cos(\x)}, {2.7+2.05*sin(\x)});
\fill (xp) circle [radius = 1pt] node [left] {$x_r'$};

%bottom triangle
\draw[-] (xr) -- (.48,1);
%\node[draw=none] at (0, .5) {$\sqrt{\gamma r}$};
\draw[-] (.48,-.1) -- (.48,1);
%\node[draw=none] at (.6, .5) {$3r$};
%\node[draw=none] at (.6, -.5) {$\sqrt{\gamma r - 9r^2}$};
%\draw[-latex] (.5, -.3) -- (.24, -.05);

%top triangle
\draw[-] (xp) -- (1.1,1);
%\node[draw=none] at (0.2, 1.7) {$\gamma - 2r$};
\draw[-] (xp) -- (1.1, 2.7);
%\node[draw=none] at (.55, 2.9) {$\sqrt{2r\gamma - 5r^2}$};
\draw[-] (1.1, 1) -- (1.1, 2.7);
%\node[draw=none] at (1.4, 1.8) {$\gamma - 3r$};

%label balls
\node[draw=none] at (-1.5, .5) {$B_{\sqrt{\gamma r}}(x_r)$};
\node[draw=none] at (-2.4, 2) {$B_{\gamma - 2r}(x_r')$};
\node[draw=none] at (-2.4, 1.2) {$\{(x - x_r) \cdot \vec \nu_r = -3r\}$};

%fill relevant region
\fill [blue, opacity = .3] ++(.48,1) arc (68:112:1.32) ;
%\node[draw=none] at (-1, 3.2) {$B_{\sqrt{\gamma r}}(x_r) \cap \{(x - x_r) \cdot \vec \nu_r \le -3r\}$};
%\draw[-latex] (-1,3.1) -- (-.2, 1.15);
\end{tikzpicture}
\caption{\label{fig:xrpic}Showing $\{(x - x_r) \cdot \vec \nu_r \le -3r\}$ is contained in $\Omega_{2r}$}
\end{figure}
\end{proof}
  
 \medskip
 
   Theorem~\ref{thm_1}, Corollary~\ref{cor_1} and Lemma~\ref{lem_3} enable us to compare $\tilde{v}$ and $v$ and conclude the following:

\begin{theorem}\label{main}
Let $v$ and $u$ respectively solve $(P_N)$ and $(P_g)$. Then we have, for any $T>0$, 
$$
\sup_{\Omega\times [0,T]}|v-u| \leq C(T)N^{-1/3}.
$$
\end{theorem}

% section Numerics Results (end)

\section{Additional Remarks}\label{sec:remarks}

\subsection{Examples in one dimension} % (fold)
\label{sub:Heat Equation Example}

First we verify that the full divergence-form drift is necessary in $(P_N)$.
\begin{theorem}\label{decay}
Let $v(x,0)$ be the characteristic function of $\Omega:=[0,1]$, let $\Phi(x)=d^3(x, [0,1])$ and let $v(x,t)$ solve the following equation with initial data $v(x,0)$:
\begin{equation}\label{drift}
v_t - v_{xx} -N\Phi_x v_x = 0 \hbox{ in } \R\times (0,\infty). 
\end{equation}
Then for any $\delta > 0$, there are $N_0$ and $T_0$ that only depend on $\delta$ so that for $N > N_0, v(x,T_0) < \delta$ in $[0,1]$.
\end{theorem}
The solution $u$ of $(H)$ with $a=0, b=1$ and $u_0 = v(x,0)$ is the stationary solution $u \equiv 1$. Thus the above theorem demonstrates in particular that $v$ does not converge to $u$ in $\R^n\times [0,T]$ as $N\to\infty$, if $T$ is chosen large.

\begin{proof}
Fix $\delta > 0$.  Note that $\phi(x):= e^{-N\Phi(x)}$ is a supersolution to \eqref{drift}, and thus $v \leq \phi$. Set $N_0$ so that $\phi(x) \leq \delta$ for $x \in \{-1, 2\}$, and let us compare $v(x,t)$ with a barrier $h(x,t)$ in $[-1, 2]\times [0,T_0]$, where  
$$
h(x,t) = 1+\delta/2- \delta (x-1/2)^2 - \delta t.
$$
Here $T_0 = \delta^{-1} - 13/4$ satisfes $h(x,T_0) = \delta$ for $x \in \{ -1, 2\}$. Since $h_t - h_{xx} \geq 0$  and $h_x \Phi_x \leq 0$, it follows that $h$ is a supersolution of \eqref{drift}, and since $h\geq \delta$ in $\{-1, 2\}\times [0,T_0]$ it follows from the comparison principle that $v \leq h$ in $[-1, 2]\times [0,T_0]$. Thus $v(x,T_0) \leq h(x,T_0)\leq 4\delta$. 

\end{proof}

 Next we show that for this penalization scheme and choice of initial data \eqref{eqn:v0}, the convergence rate of $N^{-1/3}$ given in Theorem~\ref{main:thm} is optimal.  The rate is connected to the cubic growth of $\Phi = d(x, \Omega)^3$.  The idea is that $O(N^{-1/3})$ mass leaks out as seen by the size of mass on the outside of the stationary solution $e^{-N \Phi(x)}$.  Our attempt to add additional mass in $v_0$ need not exactly cancel the mass loss, as the following example shows.
  Consider $v$ solving $(H_N)$ in $\R\times (0,\infty)$ with initial data
\begin{align}\label{eqn:cosv0}
	v_0(x) = \left\{
	     \begin{array}{lr}
			\cos(2 \pi x) + 1 & \mbox{ if } x \in [0,1] \\
			2e^{N x^3}	& \mbox{ if } x < 0\\
			2 e^{-N (x-1)^3} & \mbox{ if } x > 1.
			 \end{array}
		\right.
\end{align}
With this $v$ we have the following theorem:
 \begin{theorem}\label{error}
Let $u(x,t)$ solve $(P_g)$ in $[0,1]\times (0,\infty)$, with initial condition $u_0(x) = \cos(2 \pi x ) + 1$. Then with $v$ as above, 
Then there exists a time $T$ so that for all $N$, 
\[\sup_{[0,1]\times[0,T]}|u(x,t)-v(x,t)| \ge CN^{-1/3}.
\]
\end{theorem}
\begin{proof}
Note that $u \to 1$ and $v \to C e^{-N \Phi}$ exponentially as $t \to \infty$, and since $\Phi$ is a uniformly convex potential except in a compact set, this rate is uniform in $N$.  By conservation of mass, we must have that
\begin{align*}
	C\left[\int_{-\infty}^0  e^{Nx^3}\;dx + 1+  \int_1^\infty e^{-N (x-1)^3} \; dx \right]= 1 + \int_{-\infty}^0 2e^{N x^3}\;dx + \int_1^\infty 2e^{-N (x-1)^3} \;dx.
\end{align*}
Then denoting $I = \int_0^\infty e^{-u^3} \; du$, we have that $\int_0^\infty e^{-Nx^3}\;dx = N^{-1/3} I$, so
\begin{align*}
	C = \frac{1 + 4N^{-1/3} I}{1+2N^{-1/3} I}.
\end{align*}
Thus since $u(x,t) \to 1$, we find that in $[0,1]$, as $t \to \infty$,
\begin{align*}
	v(x,t) - u(x,t) &\to C - 1 = \frac{1 + 4N^{-1/3} I}{1+2N^{-1/3} I}- 1 = \frac{2N^{-1/3} I }{1 +2 N^{-1/3} I} \ge N^{-1/3}I.
\end{align*}
Then we can find a time $T$, independent of $N$, so that $v(x,T) - u(x,T) > N^{-1/3}I / 2$.
\end{proof}

\subsection{Stability of the drift potential } % (fold)
\label{sub:Non-normal Phi}
Here we consider the potential whose gradient does not exactly line up with $\vec \nu$ near the boundary.  For simplicity we will restrict this to the case of the heat equation and only work in $\Omega'$; the divergence case is similar.  We will write our new drift as the old drift plus a perturbation term $\Psi$, that is, we deal with the equation 
$$\label{H'_N}
\left\{
     \begin{array}{l}
	v_t = \Delta v + N \nabla \cdot(v \nabla \Phi) + N \nabla \cdot (v \nabla \Psi) \mbox{ in } \Omega', \\
	v(x,0) = v_0(x) := \left\{
	     \begin{array}{lr}
			u_0(x)	& \mbox{ in } \Omega.\\
			  u_0(S(x)) e^{-N(\Phi(x)  + \Psi(x))}& \mbox{ in } \Omega' \backslash \Omega
			 \end{array}
	\right.
		 \end{array}
		\right.\leqno{(H'_N)}
$$
Here $S(x) = x - d(x) \nabla d(x)$ in the case $A \equiv Id$.  
\begin{theorem}\label{thm:non-normal}
Suppose $\Omega$ is $C^2$, $\Omega'$ satisfies \eqref{eqn:omegabd} and further $d(\Omega, \Omega') < 1$.  Also, suppose $|\nabla \Psi| < d(x,\Omega)^3$.  Then if $u$ solves $(D)$ with $A \equiv Id$ (that is, the heat equation), and $v$ solves $(H_N')$,  we have
\begin{align*}
	\norm{u - v}_{L^{\infty}(\Omega \times [0,T])} < C(u, \Omega)(T+1) N^{-1/3}.
\end{align*}
\end{theorem}

\begin{proof}
The proof is essentially the same as in Theorem~\ref{thm:neumannapprox2dthm}, with the main difference being we apply the transform
\begin{align*}
	\varphi(x,t) = f(x,t) e^{-N(\Phi(x) + \Psi(x))}.
\end{align*}
Then using the fact that $|\nabla \Psi|$ is an order smaller than $|\nabla \Phi|$ makes it so that the extra terms in the transform are not problematic in the extension process.
 \end{proof}

\begin{appendix}
\section{ Appendix: Constructing $A(x,t)$} % (fold)
\label{sec:appendixA}
\begin{lemma}
Consider a $C^k$ domain $\Omega$ and smooth vector field $\vec{v}(x,t)$ satisfying $\vec{v}(x,t) \cdot \vec \nu(x) \ge c_0$ for all $x$ and $t$, where $\vec \nu$ is the outer unit normal to $\Omega$ at $x$.  Then there exists a symmetric $C^k$ matrix $A(x,t)$ satisfying the property that
\begin{align*}
	A(x,t) \vec \nu(x) = \vec{v}(x,t) \mbox{ for all } x \in \partial \Omega.
\end{align*}
Moreover, $A$ satisfies the ellipticity condition \eqref{eqn:strictell}. 
\end{lemma}
\begin{proof}
We start by considering an orthonormal basis of $T_x(\partial \Omega)$ written as $\{v_1(x), \dots, v_{n-1}(x)\}$, where the $v_i$ are $C^k$ in $x$.  After a rescaling $\vec{v}$ can be written as 
\begin{align*}
	\vec{v}(x,t) = \vec \nu(x) + \sum_{i = 1}^{n-1} \alpha_i(x,t) v_i(x),
\end{align*}
where the $\alpha_i$ are $C^k$ and bounded.  Then we define $S$ as
\begin{align*}
	S(x) = (\vec \nu(x), v_1(x), \dots, v_{n-1}(x)).
\end{align*}
Now we consider $A$ of the form $S B S^{-1}$, where
\begin{align*}
B &= \begin{pmatrix}
1 & \alpha_1 & \alpha_2 & \dots & \alpha_{n-1} \\
\alpha_1 & c & 0 & \dots & 0 \\
\alpha_2 & 0 & c & \dots & 0 \\
\vdots & 0& & \ddots & 0 \\
\alpha_{n-1} & 0 & 0 & \dots & c \\
\end{pmatrix} .
\end{align*}
Here $c$ is a constant to be chosen large.  We claim that if $c$ is large enough, all the eigenvalues of $B$ and hence $A$ will be uniformly positive.  This is because cofactor expansion gives that
\begin{align*}
	\det(B - \lambda I ) &= (1-\lambda) (c-\lambda)^{n-1} + \sum_{i = 1}^{n-1} (-1)^i \alpha_i^2 (c-\lambda)^{n-2} \\
	&= (c-\lambda)^{n-2} \left[ (1-\lambda) (c-\lambda) + \sum_{i = 1}^{n-1} (-1)^i \alpha_i^2\right]\\
	&= (c-\lambda)^{n-2} \left[ \lambda^2 - 2 \lambda c + c+ \sum_{i = 1}^{n-1} (-1)^i \alpha_i^2\right].
\end{align*}
Then this has eigenvalues $\lambda = c$ and writing $\beta := \sum_{i = 1}^{n-1} (-1)^i \alpha_i^2$, 
\begin{align*}
	\lambda &= \frac{ 2c \pm \sqrt{ 4c^2 - 4(c + \beta)}}{2}.
\end{align*}
Thus taking $c$ large with respect to $\beta$ ensures that all eigenvalues can be bounded by
\begin{align*}
	\lambda_0 < \lambda_i(x) < \Lambda
\end{align*}
for all $i$ and $x \in \partial \Omega$.  Thus $A$ satisfies the ellipticity condition, is symmetric, and is smooth.
	\end{proof}

%%%%%%%%%%%%%%%%%%%%%%%%%%%%%%%%%%%%%%%%%%%%%%%%%%%%%%%%%%%%%%%%%%%%
%APPENDIX B
%%%%%%%%%%%%%%%%%%%%%%%%%%%%%%%%%%%%%%%%%%%%%%%%%%%%%%%%%%%%%%%%%%%%
\section{ Appendix: Remarks on the distance functions}\label{appendixa}
In this section we prove Lemma~\ref{lem:omegalemma}.  We keep the notation $\gamma$ for the lower bound on the radii of the exterior and interior balls of $\partial \Omega$, and $\Lambda$ is the constant so that $I \le A(x,t) \le \Lambda I$ uniformly. 

\begin{lemma}
Suppose that 
\begin{align}\label{eqn:omegaconditionA}
	d(\Omega', \Omega) <\min\left[\gamma, \frac{\gamma}{\sqrt{\Lambda^2 - 1}}\left(\Lambda - \sqrt{\Lambda^2 - 1}\right) \right],
\end{align}
and $A$ is $C^2$.  Then in $\Omega'$, the distance function $d(x,\Omega)$ is $C^2$, $S(x,t)$ is well defined, and $\tilde d \lesssim d$. Further, for $x \in \Omega' \backslash \Omega$, $A \nabla d |_x \notin T_{S(x,t)} \partial \Omega$.  Lastly, $\tilde d$ is also $C^2$, and hence $S$ is $C^2$ as well, with 
\begin{align}\label{eqn:nabladA}
\nabla \tilde d(x,t)^T &= \frac{ \nabla d(S(x,t))^T \left[I - \tilde d(x,t) \nabla A(x,t) \nabla d(x) - \tilde d(x,t) A(x,t) D^2d(x)\right]}{\nabla d(S(x,t)^T) A(x,t) \nabla d(x)}.
\end{align}
\end{lemma}

\begin{proof}
The regularity of the distance function is shown in \cite{gilbtrud}. For the second part, consider a point $x \in \Omega' \backslash \Omega$.  Then since $x$ is contained in an exterior ball of $\Omega$, we must have that there is a unique nearest point $y \in \partial \Omega$ at distance $d$, and at $y$, there is an interior ball $B_\gamma(z)$.  We show that starting at $x$, going in direction $-A \nabla d$ we will wind up in this interior ball, which will show that $S(x,t)$ is well-defined.  The worst case scenario is when the angle between $\nabla d$ and $A \nabla d$ is maximal, and we note we can get an upper bound since it satisfies 
\begin{align*}
	\cos \theta = \frac{\nabla d^T A \nabla d}{|\nabla d | |A \nabla d|} \ge \frac{1}{\Lambda}.
\end{align*}

We wish to show that a ray starting from $x$, deflected by a maximal $\theta$, will hit $B_\gamma(z)$ provided $d$ is small enough.  Projecting into the plane containing $\nabla d$ and $A \nabla d$, we can consider this in two dimensions; see Figure~\ref{fig:d'}.  Solving for the intersection, we find that the distance from where we hit $B_\gamma(z)$ is given by 
\begin{align}\label{eqn:d'}
	d' = \frac{ d(d+2\gamma) \Lambda}{d+\gamma + \sqrt{(d+\gamma)^2 - d(d+2\gamma)\Lambda^2}} \approx C(\gamma, \Lambda) d
\end{align}
for $d$ small.  Then in particular we hit if
\begin{align}\label{eqn:dbdd}
	(d+\gamma)^2 - d(d+2\gamma)\Lambda^2 \ge 0,
\end{align}
and it can be checked this is equivalent to requiring 
\begin{align*}
	d < \frac{\gamma}{\sqrt{\Lambda^2 - 1}}\left(\Lambda - \sqrt{\Lambda^2 -1}\right).
\end{align*}
Thus $S(x,t)$ is well-defined if we hit $B_\gamma(z)$, which is guaranteed if $\Omega'$ satisfies \eqref{eqn:omegaconditionA}. Also, since $\tilde d \le d' / |A \nabla d|$, we find $\tilde d \lesssim C(A, \Omega)$.

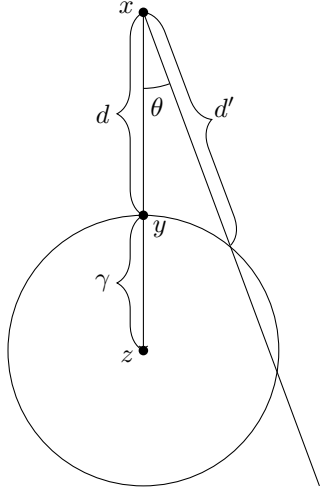
\begin{figure}
\centering
\begin{tikzpicture}[scale = 1.8]
\coordinate (x) at (0, 1.5);
\coordinate (z) at (0, -1);
\coordinate (y) at (0, 0);
%\draw[style=help lines] (-2,-2) grid (2,2);
% draw d pt
%\draw[shift=(d)] (2pt, 0pt) -- (-2pt, 0pt) node [left] {$x$};
% draw r pt
%\draw[shift=(r)] (2pt, 0pt) -- (-2pt, 0pt) node [left] {$z$};
\fill (x) circle [radius = 1pt];
\fill (z) circle [radius = 1pt];
\fill (y) circle [radius = 1pt];
\draw (x) node [left, yshift = 2pt] {$x$};
\draw (y) node [right, xshift = 0pt, yshift = -5pt] {$y$};
\draw (z) node [left, yshift = -2pt] {$z$};
%x axis
%\draw[->] (-2,0) -- (2,0) node[right] {$x$};
%y axis
\draw[->] (x) -- (z);
% circle
\draw (z) circle (1);
% line
\coordinate (e) at (1.3,-2);
\draw[-] (x) -- (e);
% draw theta node
\draw (.1, .8) node {{$\theta$}};
%draw theta angle thing
\tikzAngleOfLine(x)(e){\AngleStart}
\tikzAngleOfLine(x)(z){\AngleEnd}
\draw[-] (x)+(\AngleStart:16pt) arc (\AngleStart: \AngleEnd: 16pt);

%draw d' brace and node
\draw [decorate, decoration={brace, amplitude=10pt}] (x) -- (.65, -.22);
\draw (.6,.8) node {{$d'$}};
%draw d brace and node
\draw [decorate, decoration={brace, amplitude=10pt}] (y) -- (x);
\draw (-.3,.75) node {{$d$}};
%draw \gamma brace and node
\draw [decorate, decoration={brace, amplitude=10pt}] (z) -- (y);
\draw (-.3,-.5) node {{$\gamma$}};

\end{tikzpicture}
\caption{\label{fig:d'}Calculating an upper bound on $\tilde d$}
\end{figure}

Next, we check that $A \nabla d |_x \notin T_{S(x,t)} \partial \Omega$.  Consider the line from $x$ along $A \nabla d$ to where it hits $B_\gamma(z)$.  If $A \nabla d \in T_{S(x,t)} \partial \Omega$, then there would be an interior ball $B_\gamma(w)$ that is perpendicular to $A \nabla d$ at $S(x,t)$.  We claim $d(w,x) < d(z,x)$, which would contradict that $y$ is the nearest point to $x$ since then $d(x, \Omega) \le d(w,x) - \gamma < d(z,x) - \gamma = d(x,y)$.  This claim follows by showing that
\begin{align*}
	d(x,w)^2 = d(x,S(x,t))^2 + \gamma^2 < (d+\gamma)^2 = d(x,z)^2.
\end{align*}
But since $d(x,S(x,t)) \le d' < d(d+2\gamma)\Lambda / (d+\gamma)$ by \eqref{eqn:d'}, this is ensured if 
\begin{align*}
	\frac{d^2(d+2\gamma)^2 \Lambda^2}{(d+\gamma)^2} < (d+\gamma)^2 - \gamma^2 = d(d+2\gamma).
\end{align*}
Rearranging this yields that we need
\begin{align*}
	d(d+2\gamma) \Lambda^2 < (d+\gamma)^2,
\end{align*}
which is the same as \eqref{eqn:dbdd}, and hence true by our assumptions on the maximal size of $d$.

Finally, we use the implicit function theorem to show that $\tilde d$ is continuous.  This is because it can be given implicitly as the solution $\lambda$ to
\begin{align*}
	f(x, t, \lambda) &= d(x - \lambda A(x,t) \nabla d(x, \Omega), \Omega) = 0.
\end{align*}
Then for a given $(x_0,t_0)$, a minimal solution $\lambda$ must exist since the solution space is non-empty and everything is continuous.  Then we compute that
\begin{align*}
	\der{f}{\lambda} &= -\nabla d(x_0 - \lambda A(x_0,t_0) \nabla d(x_0))^T A(x_0,t_0) \nabla d(x_0).
\end{align*}
It can be shown that $\eqref{eqn:dbdd}$ implies that at the first point of contact to $\Omega$ along $A \nabla d$, the angle to $\Omega$ is not tangent, and hence $\der{f}{\lambda} \ne 0$.  Thus we find an implicit solution exists as we vary around $(x,t)$ around $(x_0,t_0)$, there is locally a solution $\tilde d(x,t)$, and since everything else in $f$ is $C^2$, so is $\tilde d$.  Then \eqref{eqn:nabladA} follows from the implicit function as well.

\end{proof}
\end{appendix}
\bibliographystyle{amsplain}

\bibliography{ParabolicNeumannDriftApproximation}
\end{document}